\documentclass[11pt]{article}
\usepackage{amsmath,amsfonts,amssymb,amsthm,mathrsfs}
\usepackage[a4paper,vmargin={3.5cm,3.5cm},hmargin={2.5cm,2.5cm}]{geometry}
\usepackage[font=sf, labelfont={sf,bf}, margin=1cm]{caption}
\usepackage{graphicx,graphics}
\usepackage{epsfig}
\usepackage{latexsym}
\usepackage[applemac]{inputenc}
\usepackage{ae,aecompl}
\usepackage[english]{babel}
\usepackage[colorlinks=true]{hyperref}
\usepackage[toc,page]{appendix}
\usepackage{bbm}

\date{}

\newtheorem{theo}{Theorem}[section]
\newtheorem{prop}[theo]{Proposition}
\newtheorem{cor}[theo]{Corollary}

\newtheorem{lemma}[theo]{Lemma}

\numberwithin{equation}{section}
\newcommand{\e}{{\mathrm e}}
\newcommand{\R}{{\mathbb{R}}}
\newcommand{\T}{\mathcal{T}}

\newcommand{\s}{\mathcal{S}}

\newcommand{\N}{\mathbb{N}}

\newcommand{\Z}{\mathbb{Z}}

\title{\bf \textsc{Scaling limits of k-ary growing trees}}

\author{\text{B\'en\'edicte  Haas}\thanks{ Universit\'e Paris-Dauphine and \'Ecole normale supérieure, E-mail: haas@ceremade.dauphine.fr} \ \ \& \text{Robin Stephenson}\thanks{Universit\'e Paris-Dauphine, E-mail: stephens@phare.normalesup.org}   }

\begin{document}

\maketitle

\abstract{For each integer $k \geq 2$, we introduce a sequence of $k$-ary discrete trees constructed recursively by choosing at each step an edge uniformly among the present edges and grafting on ``its middle" $k-1$ new edges. When $k=2$, this corresponds to a well-known algorithm which was first introduced by Rémy. Our main result concerns the asymptotic behavior of these trees as $n$ becomes large: for all $k$, the sequence of $k$-ary trees grows at speed $n^{1/k}$ towards a $k$-ary random real tree that belongs to the family of self-similar fragmentation trees. This convergence is proved with respect to the Gromov-Hausdorff-Prokhorov topology. We also study embeddings of the limiting trees when $k$ varies.}

\medskip

\noindent \emph{\textbf{Keywords:} random growing trees, scaling limits, self-similar fragmentation trees, Gromov-Hausdorff-Prokhorov topology}

\medskip

\noindent \emph{\textbf{AMS subject classifications:}} 60F17, 60J80

\section{Introduction}

\noindent \textbf{The model.} Let $k \geq 2$ be an integer. We introduce a growing sequence of $k$-ary trees $(T_n(k),n \geq 0)$, where $T_n(k)$ is a rooted tree with $(k-1)n+1$ leaves, constructed recursively as follows:
\begin{enumerate}
\item[$\bullet$] \textsc{Step 0}: $T_0(k)$ is the tree with one edge and two vertices: one root, one leaf.
\item[$\bullet$] \textsc{Step} $n$: given  $T_{n-1}(k)$, choose uniformly at random one of its edges and graft on ``its middle" $(k-1)$ new edges, that is split  the selected edge into two so as to obtain two edges separated by a new vertex, and then add $k-1$ new edges to the new vertex.
\end{enumerate}

\begin{figure}[ht]
\centering

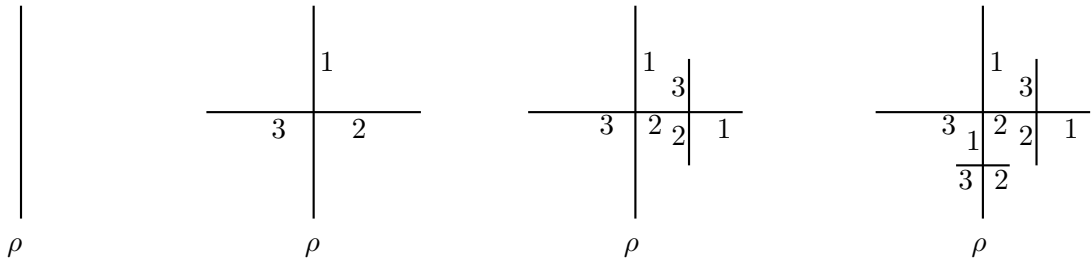
\caption{A representation of $T_n(3)$ for $n=0,1,2,3$. The edges are also labelled as explained in the paragraph just above Theorem \ref{thjoint}.}
\label{Fig1}
\end{figure}

For all $n$, this gives a tree $T_n(k)$ with indeed $(k-1)n+1$ leaves, $n$ internal nodes and $kn+1$ edges.
In the case where $k=2$, edges are added one by one and our model corresponds to an algorithm introduced by R\'emy \cite{Rem85} to generate trees uniformly distributed among the set of binary trees with $n+1$ labelled leaves. Many other dynamical models of trees growing by adding edges one by one exist in the literature, see e.g. \cite{SmyMah94,Bha07,RTV07,Ford05,CFW09}.  

\bigskip

\noindent \textbf{Scaling limits.} We are interested in the description of the metric structure of the growing tree $T_n(k)$ as $n$ becomes large.  For $k=2$, it is easy to explicitly compute the distribution of $T_n(2)$ (see e.g. \cite{Mar06}), which turns out to be that of  a (planted) binary critical Galton-Watson tree conditioned to have $2n+2$ nodes (after forgetting the order). According to the work of Aldous  on scaling limits of Galton-Watson trees \cite{Ald93}, the tree $T_n(2)$ then grows at speed $n^{1/2}$ towards a multiple of the Brownian continuum random tree (Brownian CRT). 
Let us explain this statement more formally. The trees  $T_n(2), n \geq 0$ may be viewed as metric spaces by considering that their edges are segments of length 1, and therefore belong to the set of so-called $\R$-trees. They are moreover endowed with a probability measure, the uniform probability on their leaves, which we denote by $\mu_n(2), n \geq 0$. To compare how close two such measured trees are, we use the so-called Gromov-Hausdorff-Prokhorov (GHP) topology  on the set of measured compact $\R$-trees. Background on that topic will be given in Section \ref{sec:background}. The above  result on the asymptotic behavior of the sequence $(T_n(2))$ can now be made precise as follows: there exists a compact $\R$-tree $\mathcal T_{\mathrm{Br}}$ distributed as the Brownian CRT  and a probability measure $\mu_{\mathrm{Br}}$ on $\mathcal T_{\mathrm{Br}}$ such that
\begin{equation}
\label{cvT2}
\left(\frac{T_n(2)}{n^{1/2}},\mu_n(2) \right) \ \overset{\mathrm{a.s.}}{\underset{n \rightarrow \infty}\longrightarrow} \  \left(2\sqrt 2 \mathcal T_{\mathrm{Br}}, \mu_{\mathrm{Br}}\right)
\end{equation}
for the GHP-topology. We point out that the almost sure convergence was not proved initially in \cite{Ald93}, which states, in a more general setting, convergence in distribution of rescaled Galton-Watson trees. However, it is implicit in \cite{PitmanStFl} and \cite{Mar03}. See also \cite[Theorem 5]{CH13} for an explicit statement.

Many classes of random trees are known to converge after rescaling towards the Brownian CRT. However, other limits are also possible, among which two important classes of random $\R$-trees: the class of Lévy trees introduced by Duquesne, Le Gall and Le Jan \cite{LGLJ98,DLG02,DLG05} (which is the class of all possible limits in distribution  of rescaled sequences of Galton-Watson trees \cite{DLG02}) and the class of self-similar fragmentation trees \cite{HM04,Steph13} (which is the class of  scaling limits of the so-called Markov branching trees \cite{HM12,HMPW08}). We will see in this paper that the sequence  $(T_n(k), n \geq 0)$ has a scaling limit belonging to this second category.  From now on, we will call ``fragmentation tree" any self-similar fragmentation tree, the self-similarity being implicit. Informally, a fragmentation tree with index of self-similarity $\alpha \in (-\infty,0)$ is a random compact $\R$-tree endowed with a probability measure that makes it self-similar: the subtrees  of this tree situated above a given height are distributed as the initial tree up to their own mass (with respect to the probability measure on the tree) to the power $\alpha$. These trees were introduced to code the genealogy of self-similar fragmentations, which are random processes modeling the evolution of blocks subject to splitting. We refer to \cite{BertoinBook} for background on fragmentation processes and to \cite{HM04,Steph13} for background on fragmentation trees.  In particular, it is known that the distribution of such a tree is characterized by three parameters: the index of self-similarity $\alpha$, an erosion coefficient $c \geq 0$ which corresponds to a continuous melting of the blocks and a so-called dislocation measure which is $\sigma$-finite  on the set of decreasing positive sequences with sum less than one. The role of this measure is to code the way sudden dislocations occur in the fragmentation process, or, in terms of trees, the way the relative masses of the subtrees descending from a given node are distributed. This measure  may be supported by sequences with sum strictly less than one which then means that some mass is lost during the dislocation of a fragment.  In this case, the fragmentation is called non-conservative, while it is called conservative in the other case All fragmentation trees considered in this paper have an erosion coefficient equal to 0, which is implied from now on. 
   
The Brownian CRT belongs to the family  of fragmentation trees \cite{BertoinSSF}. Its index of self-similarity is $-1/2$ and its dislocation measure $\nu^{\downarrow}_2$ is supported on the $1$-dimensional simplex $\mathcal S_2=\{\mathbf s=(s_1,s_2) \in [0,1]^2, s_1+s_2=1\}$ and defined by
$$
\nu^{\downarrow}_2({\mathrm d s_1})=\sqrt{\frac{2}{\pi}} s_1^{-3/2}s_2^{-3/2} \mathbbm 1_{\{s_1 \geq s_2\}} \mathrm d  s_1=\sqrt{\frac{2}{\pi}} s_1^{-1/2}s_2^{-1/2}\left( \frac{1}{1-s_1}+\frac{1}{1-s_2}\right)\mathbbm 1_{\{s_1 \geq s_2\}}\mathrm d s_1,
$$
where $\mathrm d s_1$ denotes the Lebesgue measure on $[0,1]$. Of course the constraint $s_1 \geq s_2$ is here equivalent to $s_1 \geq 1/2$, but we keep the first notation in view of generalizations.  

\bigskip

Our main goal is to generalize the convergence (\ref{cvT2}) to the sequences of trees $(T_n(k), n \geq 0)$ for all integers $k \geq 2$. Let $\mathcal S_k$ be the closed $(k-1)$-dimensional simplex and its variant $\mathcal S_{k,\leq}$ of dimension $k$ obtained by allowing the sum to be less than 1,
$$\mathcal S_k=\left\{\mathbf s=(s_1,s_2,...,s_k) \in [0,1]^k : \sum_{i=1}^k s_i=1\right \}; \quad \mathcal S_{k,\leq}=\left\{\mathbf s=(s_1,s_2,...,s_k) \in [0,1]^k : \sum_{i=1}^k s_i \leq 1\right \}.$$
Both spaces are endowed with the distance $$d_k(\mathbf s,\mathbf s')=\sum_{i=1}^k |s_i-s_i'|,$$ which makes them compact.
The Lebesgue measure on $\s_k$ can be written as $\mathrm d\mathbf{s}=\prod_{i=1}^{k-1}\mathrm ds_i$, with $s_k$ being implicitly defined by $1-\sum_{i=1}^{k-1}s_i$, whereas that on $\mathcal S_{k,\leq}$ should be understood as 
$\mathrm d\mathbf{s}=\prod_{i=1}^{k}\mathrm ds_i.$

\begin{theo}
\label{thmain} Let $\mu_n(k)$ be the uniform measure on the leaves of $T_n(k)$.
There exists a  $k$-ary  $\R$-tree $\mathcal T_k$, endowed with a probability measure on its leaves $\mu_k$, such that
$$
\left(\frac{T_n(k)}{n^{1/k}}, \mu_n(k) \right)\ \overset{\mathbb{P}}{\longrightarrow} \ \left(\mathcal T_{k},\mu_k\right)
$$
for the GHP-topology.
The measured tree  ($\mathcal T_k, \mu_k)$ belongs to  the family of conservative fragmentation trees, with index of self-similarity $-1/k$. Its dislocation measure $\nu^{\downarrow}_k$ is supported on  
$\mathcal S_k$ and defined by
$$
\nu^{\downarrow}_k({\mathrm d \mathbf s})=\frac{(k-1)!}{k(\Gamma(\frac{1}{k}))^{k-1}} \prod_{i=1}^k s_i^{-(1-1/k)}\left(\sum_{i=1}^k \frac{1}{1-s_i}\right)\mathbbm 1_{\{s_1\geq s_2 \geq ... \geq s_k\}}\mathrm d \mathbf s,
$$  
where $\Gamma$ stands for Euler's Gamma function.
\end{theo}

Note that the convergence is a little weaker than (\ref{cvT2}) since it is only a convergence in probability. However the  finite dimensional marginals of $T_n(k)$ converge almost surely as we will see later in Proposition \ref{prop:cvmarginalesGH}. 
Note also that  $\nu^{\downarrow}_k$ is a $\sigma$-finite measure on $\mathcal S_k$ such that $$\int_{\mathcal S_k}(1-s_1) \nu^{\downarrow}_k(\mathrm d \mathbf s)<\infty$$ but with infinite total mass. This fact implies in particular that the leaves of the tree $\mathcal T_k$ are dense in $\mathcal T_k$ (see \cite[Theorem 1]{HM04}).

\vspace{0.05cm}

Since the limiting tree is a fragmentation tree, we immediately have its Hausdorff dimension. Indeed, the Hausdorff dimension of conservative fragmentation trees was computed in \cite{HM04}  and this result was extended to general fragmentation trees in \cite{Steph13}. In particular, we know from \cite[Theorem 2]{HM04} that the Hausdorff dimension of a conservative fragmentation tree with index of self-similarity $\alpha<0$ and dislocation measure $\nu$ (and no erosion) is equal to $\max(|\alpha|^{-1},1)$ provided that the measure $\nu$ integrates $(s_1^{-1}-1)$ on the set of decreasing sequences with sum one. Here,  $$\int_{\mathcal S_k}(s_1^{-1}-1)\nu_k^{\downarrow}(\mathrm d \mathbf s)\leq \int_{\mathcal S_k}k^{-1}(1-s_1)\nu_k^{\downarrow}(\mathrm d \mathbf s)<\infty$$ since $s_1 \geq s_2 \geq ... \geq s_k$ together with $\sum_{i=1}^k s_i=1$ implies that $s_1 \geq 1/k$.

\begin{cor}
The Hausdorff dimension of tree $\mathcal T_k$ is almost surely $k$.
\end{cor}

\medskip

%\begin{rem} 
\noindent \textit{Remark.}
From the recursive construction of the sequence $(T_n(k))$ one could believe at  first sight that the trees $T_n(k), n \geq 0$, as well as their continuous counterparts $\mathcal T_k$, are invariant under uniform re-rooting (which means that the law of the tree re-rooted at a leaf chosen uniformly at random is the same as the initial tree). Actually, this is only true for $k=2$. For $k=2$, this is a well-known property of the Brownian CRT (\cite{Ald91}). For $k \geq 3$,  it is easy to check for small values of $n$ that this property is not satisfied for $T_n(k)$. In the continuous setting, it is known that a fragmentation tree having the invariance under re-rooting property necessarily belongs to the family of stable Lévy trees (\cite{HPW09}). It is also well-known that, up to a multiplicative scaling, the unique stable Lévy tree without vertices of infinite degree is the Brownian CRT (\cite{DLG05}). Hence $\mathcal T_k$ is not invariant under uniform re-rooting for $k \geq 3$.
%\end{rem}

\bigskip

\noindent \textbf{Labels on edges and subtrees.} Partly for technical reasons, we want to label all the edges of $T_n(k)$, with the exception of the edge adjacent to the root, with integers from 1 to $k$ (see Figure \ref{Fig1} for an illustration). We do this recursively. The unique edge of $T_0(k)$ has no label since it is adjacent to the root. Given $T_n(k)$ and its labels, focus on the new vertex added in the middle of the selected edge. This edge was split into two: one new edge going towards the root, the other going away from it. Have the edge going towards the root keep the original label of the selected edge (no label if it is adjacent to the root), and have the other one be labelled $1$. The $k-1$ new edges added after that will be labelled $2,\ldots,k$, say uniformly at random (actually, the way  these $k-1$ additional edges are labelled is not important for our purpose, but the index 1 is important).

%\noindent \textbf{Joint convergence.} 
Now fix $2 \leq k'<k$. We consider, for all $n$, the  $k'$-ary  subtree of $T_{k}(n)$ obtained by discarding all edges with label larger than or equal to  $k'+1$ as well as their descendants. This subtree is denoted by $T_{k,k'}(n)$. We are interested in the sequence of subtrees $(T_{k,k'}(n), n \geq 0)$, because up to a (discrete) time-change in $n$, it is distributed as the sequence $(T_{k'}(n),n\geq 0$) (see Lemma \ref{Lemmalaw}). As a consequence, we will see that a rescaled version of $\mathcal T_{k'}$ is nested in $\mathcal T_k$. Moreover this version can be identified as a non-conservative fragmentation tree. All this is precisely stated in the following theorem.

\begin{theo} 
\label{thjoint}
For each $n \in \mathbb Z_+$, endow $T_n(k)$ with the uniform probability on its leaves $\mu_n(k)$ and $T_n(k,k')$ with the image of this probability by the projection on $T_n(k,k')$. This image measure is denoted by $\mu_n(k,k')$. Then
$$
\left(\left(\frac{T_n(k)}{n^{1/k}}, \mu_n(k) \right), \left(\frac{T_n(k, k')}{n^{1/k}}, \mu_n(k,k')\right)\right) \ \overset{\mathbb P}{\underset{n \rightarrow \infty} \longrightarrow} \ \left((\mathcal T_k ,\mu_k),( \mathcal T_{k,k'}, \mu_{k,k'}) \right)
$$ 
for the GHP-topology, where $\mathcal T_{k,k'}$ is a closed subtree of $\mathcal T_k$ and $(\mathcal T_{k,k'}, \mu_{k,k'})$ has the distribution of a non-conservative fragmentation tree with index $-1/k$. Its dislocation measure $\nu^{\downarrow}_{k,k'}$ is supported on $\mathcal S_{k',\leq}$ and defined by $$\nu^{\downarrow}_{k,k'}({\mathrm d \mathbf s})=\frac{(k'-1) !}{k(\Gamma(\frac{1}{k}))^{k'-1} \Gamma(1-\frac{k'}{k})}\times \frac{1}{(1-\sum_{i=1}^{k'} s_i)^{k'/k}} \prod_{i=1}^{k'} s_i^{-(1-1/k)} \left(\sum_{i=1}^{k'} \frac{1}{1-s_i}\right) \mathbbm 1_{\{s_1\geq ...  \geq s_{k'}\}} \mathrm d\mathbf{s}.$$
Moreover,
\begin{equation}
\label{lawsubtree}
\mathcal T_{k,k'} \overset{\mathrm{(d)}}= M^{1/k'}_{k'/k,1/k} \cdot \mathcal T_{k'} 
\end{equation}
where in the right side, $M_{k'/k,1/k}$ has a generalized Mittag-Leffler distribution with parameters $(k'/k,1/k)$ and is independent of $\mathcal T_{k'} $.
\end{theo}

The identity (\ref{lawsubtree}) is similar to results established in \cite{CH13} on the embedding of stable Lévy trees. The precise definition of generalized Mittag-Leffler distribution will be recalled in Section \ref{sec:stack}.  In that section we will also see how to extract a random rescaled version of $\mathcal T_{k'} $ directly from the limiting fragmentation tree $\mathcal T_k$  by adequately pruning subtrees on each of its branch point (Proposition \ref{prop:prun}). 

\bigskip

\noindent \textbf{Organization of the paper.} After having recalled background on $\R$-trees and the GHP metric in Section \ref{sec:background}, we will use two approaches to prove our results. The first one, developed in Section \ref{sec:cvd}, consists in checking that the sequence $(T_n(k),n\geq 0)$ possesses the so-called Markov branching property and then use results of Haas and Miermont \cite{HM12} on scaling limits of Markov branching trees to obtain the convergence in distribution of the rescaled trees $(T_n(k))$ towards a fragmentation tree. Our second approach, in Section \ref{sec:cvproba},  is based on urn schemes and the Chinese restaurant process of Pitman \cite{PitmanStFl}. It provides us the convergence in probability of the rescaled trees $(T_n(k))$ towards a compact $\R$-tree, but does not allow us to identify the limiting tree as a fragmentation tree. Combination of these two approaches then fully proves Theorem \ref{thmain}. In Section \ref{sec:cvproba}, we also treat the convergence in probability of the rescaled subtrees $(T_n(k,k'))$. The distribution of the limit will be identified in Section \ref{sec:stack}, hence giving the convergence results of Theorem \ref{thjoint}. Lastly, still in Section \ref{sec:stack}, we study the embedding of the limiting trees $\mathcal T_k$ as $k$ varies: for all $k'<k$, we show how to extract directly from  $\mathcal T_k$ a tree with the distribution of $\mathcal T_{k,k'}$ and prove the relation (\ref{lawsubtree}).

\bigskip

\medskip

\textbf{From now on, $k$ and $k'$ are fixed, with $2 \leq k' < k$. To lighten notation, we will use, up until Section \ref{sec:joint}, $T_n$ instead of $T_n(k)$ and $T'_n$ instead of $T_n(k,k')$.}

%%%%%%%%%%%%%%%%%%%%%%%%%%%%%%%%%%%%%%%%%%%%%%%%%%%%%%%%%%
\section{Background on $\mathbb R$-trees and GHP-topology}
\label{sec:background}
%%%%%%%%%%%%%%%%%%%%%%%%%%%%%%%%%%%%%%%%%%%%%%%%%%%%%%%%%%

%\noindent \textbf{$\mathbb R$-trees,  GHP topology.} 
We briefly recall background on $\R$-trees (or real trees) and Gromov-Hausdorff-Prokhorov distance, and refer to \cite{Eva08,LG06} for an overview on this topic.

An $\R$-tree is a metric space $(\T,d)$ such that, for any points $x$ and $y$ in $\T$, there exists a geodesic path from $x$ to $y$ and, up to time-reparametrization, this is the only continuous self-avoiding path from $x$ to $y$. We denote by $[[x,y]]$ this geodesic, and also write $]]x,y]]$ or $[[x,y[[$ when we want to exclude $x$ or $y$. Our trees will always be rooted at a point $\rho\in\T$. The height of a point $x\in\T$ is defined as $ht(x)=d(x,\rho)$ and the height of the tree itself is the supremum of the heights of its points. The set of descendants of $x$, called $\T_x$, is the set of all $y\in\T$ such that $x\in [[\rho,y]]$. The degree of $x$ is the number of connected components of $\mathcal T \backslash \{x\}$. We call leaves of $\T$ all the points which have degree $1$, excluding the root. A $k$-ary tree is a tree  whose points have degrees in $\{1,2,k+1\}$ (with at least one point of degree $k+1$).
Given two points $x$ and $y$, we define  $x\wedge y$ as the unique point of $\mathcal T$ such that $[[\rho,x]]\cap [[\rho,y]]=[[\rho,x \wedge y]]$. It is called the branch point of $x$ and $y$ if its degree is larger or equal to 3. For $a>0$, we define the rescaled tree $a\T$ as $(\T,ad)$ (the metric $d$ thus being implicit and dropped from the notation). Finally, note that any graph-theoretical tree can be viewed as an $\R$-tree by considering each edge as a line segment with an arbitrarily chosen length, usually $1$.

Recall that, if $A$ and $B$ are two nonempty compact subsets of a metric space $(E,d)$,  the Hausdorff distance between $A$ and $B$ is defined by 
	\[d_{E,\mathrm{H}}(A,B)=\inf \big\{\varepsilon>0\:;\: A\subset B^\varepsilon\text{ and }B\subset A^\varepsilon \big\},
\]
where $A^{\varepsilon}$ and $B^{\varepsilon}$ are the closed $\varepsilon$-enlargements of $A$ and $B$. The Gromov-Hausdorff convergence generalizes this and allows us to talk about convergence of compact $\R$-trees. Given two compact rooted trees $(\T,d,\rho)$ and $(\T',d',\rho')$, let
	\[d_{\mathrm{GH}}(\T,\T') = \inf \big[ \max (d_{\mathcal{Z},\mathrm{H}} (\phi(\T),\phi'(\T')), d_\mathcal{Z}(\phi(\rho),\phi'(\rho')))\big],
\]
where the infimum is taken over all pairs of isometric embeddings $\phi$ and $\phi'$ of $\T$ and $\T'$ in the same metric space $(Z,d_{Z}),$ for all choices of metric spaces $(Z,d_{Z})$. We will also be concerned with \textit{measured} trees, that is  $\R$-trees equipped with a probability measure on their Borel sigma-field (it is implicit from now on that in this paper a measure on a metric space is actually a Borel measure). To this effect, recall first the definition of the Prokhorov distance between two probability measures $\mu$ and $\mu'$ on a metric space $(E,d)$:
	\[d_{E,\mathrm P}(\mu,\mu')=\inf \big\{\varepsilon >0\:;\: \forall A\in\mathcal{B}(E), \mu(A)\leq\mu'(A^{\varepsilon})+\varepsilon\text{ and } \mu'(A)\leq\mu(A^{\varepsilon})+\varepsilon \big\}.
\]
Now, given two measured compact rooted trees $(\T,d,\rho,\mu)$ and $(\T',d',\rho',\mu')$, we let
  \[d_{\mathrm{GHP}}(\T,\T') = \inf \big[ \max (d_{\mathcal{Z},\mathrm{H}} (\phi(\T),\phi'(\T')), d_\mathcal{Z}(\phi(\rho),\phi'(\rho')),d_{\mathcal{Z}, \mathrm P}(\phi_*\mu,\phi'_*\mu')\big],
\]
where the infimum is taken on the same space as before and $\phi_*\mu$, $\phi'_*\mu'$ are the push-forwards of $\mu$, $\mu'$ by $\phi$, $\phi'$.

As shown in \cite{EPW06} and \cite{ADH}, the space of compact rooted $\R$-trees (respectively compact measured rooted $\R$-trees), taken up to root-preserving isomorphisms (resp. root-preserving and measure-preserving) and equipped with the GH (resp. GHP) metric is Polish. In this paper, we will implicitly identify two (measured) rooted $\mathbb R$-trees when their are isometric and still use the notation $(\mathcal T,d)$ (or $\mathcal T$ when the metric is clear) to design their isometry class.
Typically, the GHP convergence of a sequence is shown by exhibiting a specific embedding of our trees in the space $\ell^1$ of summable sequences, under which we have Hausdorff convergence of the trees and Prokhorov convergence of the measures.

%%%%%%%%%%%%%%%%%%%%%%%%%%%%%%%%%%%%%%%%%%%%%%%%%%%%%%%%%%
\section{Convergence in distribution and identification of the limit}
\label{sec:cvd}
%%%%%%%%%%%%%%%%%%%%%%%%%%%%%%%%%%%%%%%%%%%%%%%%%%%%%%%%%%

In this section, we use \cite[Theorem 5]{HM12} on scaling limits of Markov branching trees to get the convergence in distribution of the rescaled trees $n^{-1/k}T_n$ and identify the limit distribution. Actually, the method used in the following section will yield a stronger convergence, convergence in probability, but that approach does not allow us to identify the distribution of the limit.  We will also set up here some material needed to identify the distribution of the limit of the subtrees $n^{-1/k}T_n'$. The convergence of these subtrees will be proved in the next section, and the limit will then be identified in Section \ref{sec:stack}.  

Let $n\in\mathbb Z_+$ and consider the tree $T_{n+1}$. Its root is connected to only one edge, after which there are $k$ subtrees. These subtrees can be identified by the label given to their first edge, and we call them $(T^i_n)_{i\leq k}$, where $i \leq k$ refers to the edge labelled $\mathrm{n}°i$ (implicitly, $i \geq 1$ here). For all $i\leq k$, we let $X^i_n$ be the number of internal nodes of $T^i_n$ and we let $q_n$ be the distribution of $(X^i_{n})_{i\leq k}$ seen as an element of $$\mathcal C^k_n=\left\{ \lambda=(\lambda_1,...,\lambda_k) \in \mathbb Z_+^k : \sum_{i=1}^k \lambda_i=n \right\}.$$
%We then denote by $q_n^{\downarrow}$ the distribution on the set  $\mathcal P_n$ of partitions of $n$ obtained by considering the decreasing rearrangement of the $X^i_{n},i\in[k]$.
To use the results of \cite{HM12}, we have to check 
\begin{enumerate}
\item[(i)] that the sequence $(T_n)$ is Markov branching, which roughly means that conditionally on their sizes, the trees $T^i_n, i\leq k,$ are mutually independent and have, respectively, the same distribution as $T_{X^i_n}, i\leq k$; 
\item[(ii)] that appropriately rescaled, the distribution $q_n$ converges. 
\end{enumerate}
We start by studying this probability $q_n$ in Section \ref{sec:asympqn} and then prove the Markov branching property and get the limit distribution in Section \ref{sec:conclusion}.

%%%%%%%%%%%%%%%%%%%%%%%%%%%%%%%%%%
\subsection{Description and asymptotics of the measure $q_n$}
\label{sec:asympqn}
%%%%%%%%%%%%%%%%%%%%%%%%%%%%%%%%%%

Let $\bar{q}_n$ be the distribution of $(X^i_n/n)_{i\leq k}$, it is a probability measure on $\s_k$, $\forall n \geq 1$. As we will see below in Proposition \ref{prop:cvmes}, the continuous scaling limit of these distributions is  the measure $\nu_k$ on ${\s}_k$ defined by
	\[\nu_k(\mathrm d \mathbf{s})=\frac{1}{k(\Gamma(\frac{1}{k}))^{k-1}}\frac{1}{1-s_1} \prod_{i=1}^k s_i^{-(1-1/k)} \mathrm d\mathbf{s}.
\]
Note the dissymmetry between the index 1 and the others. This is due the fact that the subtree $T_n^1$ is often much larger than the other ones, since, in the $n$-th step of the recursive construction, in the case where the new $k-1$ edges are added on the edge adjacent to the root, the subtree $T_n^1$ has $n$ internal nodes whereas the $k-1$ other ones have none.

Since we are also interested in describing the asymptotic behaviour of the subtrees $T'_n$, we will also need to consider, for $n \geq 1$, the probability measures $\overline q'_n$ on $S_{k',\leq}$ obtained by considering the first $k'$ elements of $(X^i_n/n)_{i\leq k}$. Their continuous scaling limit (see Corollary \ref{coro:cvmes}) is denoted by $\nu_{k,k'}$ and defined on $\mathcal S_{k',\leq}$ by
$$
\nu_{k,k'}(\mathrm d\mathbf{s})=\frac{1}{k(\Gamma(\frac{1}{k}))^{k'-1} \Gamma(1-k'/k)}\times \frac{1}{\left(1-s_1\right)(1-\sum_{i=1}^{k'} s_i)^{k'/k}} \prod_{i=1}^{k'} s_i^{-(1-1/k)} \mathrm d\mathbf{s}.
$$
For $\mathbf{s}\in \s_k$, we let $\mathbf{s}^{\downarrow}$ be the sequence obtained by reordering the elements of $\mathbf s$ in the decreasing order. This map is continuous from $\s_k$ to $\s_k$. For any probability measure $\mu$ on $\s_k$, then let $\mu^{\downarrow}$ be the image of $\mu$ by it. 

\bigskip

\noindent \textbf{Examples.} For instance, one can check that  the measure  $\nu^{\downarrow}_k$ associated to $\nu_k$ indeed coincides with the definition of $\nu^{\downarrow}_k$ in Theorem \ref{thmain}. And similarly for the measure $\nu^{\downarrow}_{k,k'}$ and its expression in Theorem \ref{thjoint}.

\bigskip

The main goal of this section is to prove the following result. 

\begin{prop}
\label{prop:cvmes}
We have the following weak convergence of measures on ${\s}_k$:
	\[n^{1/k} (1-s_1)\bar{q}_n(\mathrm d \mathbf{s}) \underset{n\to\infty} \Rightarrow  (1- s_1)\nu_k(\mathrm d \mathbf{s}).  \\
\]
As a consequence, \[n^{1/k} (1-s_1)\bar{q}_n^{\downarrow}(\mathrm d \mathbf{s}) \underset{n\to\infty}\Rightarrow (1-s_1)\nu_k^{\downarrow}(\mathrm d \mathbf{s}).
\]
\end{prop}

The \emph{symmetric Dirichlet measure} on ${\s}_k$ with parameter $k^{-1}$ is  $\Gamma(1/k)^{-k}(\prod_{i=1}^k s_i)^{-(1-1/k)} \mathrm d\mathbf{s}$. It is well-known and easy to check that this defines a probability measure on ${\s}_k$. As a direct consequence, we see  that 
\begin{equation*}
\int_{\s_k} (1-s_1)\nu_k(\mathrm d \mathbf{s}) = \frac{\Gamma{(1/k)}}{k}.
\end{equation*}
More generally, we will need several times in this paper the well-known fact that for any integer $K\geq 2$ and all $K$-uplets  $\alpha_1,...,\alpha_K>0$, 
\begin{equation}
\label{beta}
\int_{\s_K} \prod_{i=1}^K s_i^{\alpha_i-1} \mathrm d\mathbf{s}=\frac{\prod_{i=1}^K \Gamma(\alpha_i)}{\Gamma\big(\sum_{i=1}^K \alpha_i \big)},
\end{equation}
where $x_K=1-\sum_{i=1}^{K-1} x_i$. 

\vspace{0.1cm}
The results of Proposition \ref{prop:cvmes} can easily be transferred to $\overline q_n'$:
\begin{cor} 
\label{coro:cvmes}
We have the following weak convergences of measures on ${\s}_{k',\leq}$:
	\[n^{1/k} (1-s_1)\bar{q}'_n(\mathrm d\mathbf{s}) \underset{n\to\infty} \Rightarrow  (1- s_1) \nu_{k,k'}(\mathrm d \mathbf{s})  \\
\]
and, \[n^{1/k} (1-s_1)(\bar{q}'_n)^{\downarrow}(\mathrm d \mathbf{s}) \underset{n\to\infty}\Rightarrow (1-s_1)\nu_{k,k'}^{\downarrow}(\mathrm d \mathbf{s}).
\]
\end{cor}

In order to prove Proposition \ref{prop:cvmes}, we start by explicitly computing the measure $q_n$ in Section \ref{sec:qn}. We then set up preliminary lemmas in Section \ref{sec:premlem} and lastly turn to the proofs of Proposition \ref{prop:cvmes} and Corollary \ref{coro:cvmes} in Section \ref{sec:proofcv}.

\subsubsection{The measure $q_n$}
\label{sec:qn}

\begin{prop} For all $\lambda\in \mathcal C^k_n$,
	\[
q_n(\lambda)=\frac{1}{k(\Gamma(\frac{1}{k}))^{k-1}}\left(\prod_{i=1}^k \frac{\Gamma(\frac{1}{k}+\lambda_i)}{\lambda_i!}\right)\frac{n!}{\Gamma(\frac{1}{k}+n+1)} \left(\sum_{j=1}^{\lambda_1+1} \frac{\lambda_1!}{(\lambda_1-j+1)!}\frac{(n-j+1)!}{n!}\right).
\]
\end{prop}

\begin{proof}
Let $N_1,\ldots,N_{n+1}$ be the $n+1$ internal nodes of $T_{n+1}$, listed in order of apparition, and let $J$ be the random variable such that $N_J$ is the first node encountered after the root of $T_n$. Recall that $T^1_n,\ldots,T^k_n$ denote the ordered subtrees rooted at $N_J$. For $\lambda \in \mathcal C_n^k$ and $j \in \mathbb N$, we first compute the probability $p_j(\lambda)$ that $J=j$, $T^1_n$ contains the nodes  $N_1,\ldots,N_{j-1},N_{j+1},N_{\lambda_1+1}$, $T^2_n$ contains the nodes $N_{\lambda_1+2},\ldots,N_{\lambda_1+\lambda_2+1}$ and so on until $T^k_n$, which contains the nodes $N_{\lambda_1+\ldots+\lambda_{k-1}+2},\ldots,N_{n+1}$. This probability is null for $j>\lambda_1+1$. For $1 \leq j \leq \lambda_1+1$, since each edge is chosen with probability $1/(1+kp)$ when constructing $T_{p+1}$ from $T_p$, $p\geq 1$, we get
\begin{align*}
p_j(\lambda)&=\frac{1}{1+k(j-1)}\prod_{p=j+1}^{\lambda_1+1}\frac{1+k(p-2)}{1+k(p-1)}\prod_{i=2}^k \prod_{p=1}^{\lambda_i}\frac{1+k(p-1)}{1+k(\lambda_1+\ldots+\lambda_{i-1}+p)} \\
&=\frac{\prod_{i=1}^k \prod_{p=1}^{\lambda_i-1} (1+kp)}{\prod_{p=1}^n (1+kp)}
\end{align*}
(by convention, a product indexed by the empty set is equal to 1). Note that $p_{j}(\lambda)=p_1(\lambda)$ for all $j\leq \lambda_1+1$.
Note also that this probability does not change if we permute the indices of nodes $N_{j+1},\ldots,N_n$ (both the numerator and the denominator have the same factors, just in different orders). We thus have
\begin{align*}
q_n(\lambda)&=\sum_{j=1}^{\lambda_1+1} \frac{(n-j+1)!}{(\lambda_1-j+1)! \prod_{i=2}^k \lambda_i !}  p_1(\lambda)\\
            &=\frac{n!}{\prod_{p=1}^n (1+pk)}\prod_{i=1}^k \frac{\prod_{p=1}^{\lambda_i-1} (1+pk)}{\lambda_i !}\sum_{j=1}^{\lambda_1+1} \frac{\lambda_1!}{(\lambda_1-j+1)!}\frac{(n-j+1)!}{n!}.
\end{align*}
The proof is then ended by using the fact that $\Gamma(\frac{1}{k}+q)=\Gamma(\frac{1}{k})k^{-q} \prod_{p=0}^{q-1}(1+kp)$ for any $q \in \mathbb Z_+$.
\end{proof}

\subsubsection{Preliminary lemmas}
\label{sec:premlem}

The proof of  Proposition \ref{prop:cvmes} relies on the convergence of some Riemann sums. To set up these convergences, we first
rewrite $q_n(\lambda), n \geq 1,$ in the form
\begin{equation}
\label{qn}
q_n(\lambda)=\frac{1}{k\Gamma(\frac{1}{k})^{k-1}}\frac{\prod_{i=1}^k \gamma_k (\lambda_i)}{(n+1)\gamma_k(n+1)}\beta_{n}\left(\frac{\lambda_1}{n}\right),
\end{equation}
where, for all $x\geq0$, 
	\[\gamma_k(x)=\frac{\Gamma(\frac{1}{k}+x)}{\Gamma(1+x)}
\]
and, for all $x\in [0,1]$ and $n\in \N$, 
	\[\beta_n(x)=1+\sum_{j=1}^{\lfloor nx \rfloor} \frac{nx(nx-1)\ldots(nx-j+1)}{n(n-1)\ldots(n-j+1)}.
\]

\begin{lemma} 
\label{lem:gamma}
The following convergence of functions 
$$
x \mapsto n^{1-1/k}\gamma_k(nx) \underset{n \rightarrow \infty}\longrightarrow  x\mapsto x^{-(1-1/k)}
$$
holds uniformly on all compact subsets of $(0,1]$. Moreover, there exists  a finite constant $A$ such that $\gamma_k(x)\leq A x^{-(1-1/k)}$  for all $x\geq0$. 
\end{lemma}

\begin{proof} Pointwise convergence comes from a direct application of Stirling's formula. The uniformity of this convergence on all compact subsets of $(0,1]$ can be proved through standard monotonicity argument (sometimes known as Dini's Theorem): we only need to notice that $\gamma_k$ is a nonincreasing function of $x\geq0$. This can be done through differentiating; indeed, $\gamma_k$ is differentiable and we have, for all $x\geq0$,
	\[\gamma_k'(x)=\frac{\Gamma'(\frac{1}{k}+x)\Gamma(x+1)-\Gamma(\frac{1}{k}+x)\Gamma'(x+1)}{(\Gamma(x+1))^2}.
\]
Notice that the function $x\mapsto \Gamma'(x)/\Gamma(x)$ is nondecreasing on $(0,+\infty)$, since the Gamma function is logarithmically convex (see e.g. \cite{Art64}). Therefore, the derivative of $\gamma_k$ is indeed nonpositive. Lastly, the domination of $\gamma_k$ by a constant times  the power function $x^{-(1-1/k)}$  for all $x\geq0$ follows immediately from Stirling's formula and the fact that $\gamma_k$ is continuous on $[0,+\infty)$.
\end{proof}

\begin{lemma} 
\label{lem:beta}
The function $\beta_n$ converges uniformly to the function $x\mapsto (1-x)^{-1}$ on all compact subsets of $[0,1)$. Moreover $(1-x)\beta_n(x) \leq 1$ for all $x \in [0,1]$ and all $n \in \mathbb N$. 
\end{lemma}
\begin{proof} The proof works on the same principle as the previous one: since $\beta_n$ is obviously a nondecreasing function, we only need to show that the sequence converges pointwise. This is immediate for $x=0$, and will be done with the help of the dominated convergence theorem in the other cases. Note that, for all $x\in[0,1]$ and $j \in \mathbb N$
$$
\frac{nx(nx-1)\ldots(nx-j+1)}{n(n-1)\ldots(n-j+1)} \underset{n \rightarrow \infty}\longrightarrow x^j \quad \text{and} \quad \frac{nx(nx-1)\ldots(nx-j+1)}{n(n-1)\ldots(n-j+1)}\leq x^j, \ \forall n\in\N, j\leq \lfloor nx \rfloor,
$$
which is summable for $x\in [0,1)$. The dominated convergence theorem then ensures us that $\beta_n(x)$ converges to $\sum_{j=0}^{\infty}x^j=(1-x)^{-1}$ uniformly on all compact subsets of $[0,1)$.
\end{proof}

\begin{lemma}
\label{lem:tight}
The sequence of measures $n^{\frac{1}{k}}(1-s_1)\bar{q}_n(\mathrm d \mathbf{s})$ satisfies:
	\[\forall \varepsilon>0, \exists \eta>0, \forall n\in\N, \\
n^{\frac{1}{k}}\sum_{\lambda\in \mathcal C_n^k} \left(1-\frac{\lambda_1}{n}\right)q_n(\lambda) \mathbbm{1}_{\{\exists i, \lambda_i<\eta n\}}\: <\varepsilon.
\]
\end{lemma}

\begin{proof}
We use (\ref{qn}). By individually bounding all the instances of $\gamma_k(x)$ by $Ax^{-(1-1/k)}$ and $(1-x)\beta_n(x)$ by $1$ we are reduced to showing
	\[\forall \varepsilon>0, \exists \eta>0, \forall n\in\N, \\
\sum_{\lambda\in\mathcal C_n^k}\prod_{i=1}^{k} \lambda_i^{-(1-1/k)}\mathbbm{1}_{\{\exists i, \lambda_i<\eta n\}}\: <\varepsilon.
\]
By virtue of symmetry, we can restrict ourselves to the case where $\lambda$ is nonincreasing. The condition $\exists i, \lambda_i<\eta n$ then boils down to $\lambda_k < \eta n$. Summation over $\lambda$ nonincreasing and in $\in\mathcal C^k_n$ is done by choosing first $\lambda_k$ then $\lambda_{k-1}$, going on until $\lambda_2$, the first term $\lambda_1$ being then implicitly defined as $n-\lambda_2-\ldots-\lambda_k$. Let $\varepsilon>0$, it is enough to find $\eta>0$ such that, for any $n\in\N$,
\[\sum_{\lambda_k=1}^{\lfloor \eta n \rfloor}\sum_{\lambda_{k-1}=\lambda_k}^{\lfloor n/(k-1)\rfloor}\ldots\sum_{\lambda_2=\lambda_3}^{\lfloor n/2 \rfloor} \mathbbm{1}_{\{\lambda_1 \geq \lambda_2\}}\prod_{i=1}^{k} \lambda_i^{-(1-1/k)} \: <\varepsilon.
\]
By using $\lambda_1\geq n/k$, we obtain
\begin{align*}
\sum_{\lambda_k=1}^{\lfloor \eta n \rfloor}\sum_{\lambda_{k-1}=\lambda_k}^{\lfloor n/(k-1) \rfloor}\ldots & \sum_{\lambda_2=\lambda_3}^{\lfloor n/2 \rfloor} \mathbbm{1}_{\{\lambda_1 \geq \lambda_2\}}\prod_{i=1}^{k} \lambda_i^{-(1-1/k)} \leq \Big(\frac{n}{k}\Big)^{-(1-1/k)}\sum_{\lambda_k=1}^{\lfloor \eta n \rfloor}  \sum_{\lambda_{k-1}=1}^{\lfloor n/(k-1)\rfloor}\ldots \sum_{\lambda_2=1}^{\lfloor n/2 \rfloor} \prod_{i=2}^{k}\lambda_i^{-(1-1/k)}.
\end{align*}
Standard comparison results between series and integrals imply that, since the function $t\mapsto  t^{-(1-1/k)}$ is nonincreasing and has an infinite integral on $[1,\infty)$, there exists a finite constant $B$ such that, for all $n\geq 1$, $\sum_{j=1}^n j^{-(1-1/k)} \leq B n^{\frac{1}{k}}$.
We thus get
\begin{align*}
\sum_{\lambda_k=1}^{\lfloor \eta n \rfloor}\sum_{\lambda_{k-1}=\lambda_k}^{\lfloor n/(k-1) \rfloor}\ldots\sum_{\lambda_2=\lambda_3}^{{\lfloor n/2 \rfloor}} \mathbbm{1}_{\{\lambda_1 \geq \lambda_2\}}\prod_{i=1}^{k} \lambda_i^{-(1-1/k)} \leq B'\eta^{\frac{1}{k}} n^{-(1-1/k)} \eta^{\frac{1}{k}} (n^{\frac{1}{k}})^{k-1} \leq B'' \eta^{\frac{1}{k}}
\end{align*}
where $B'$ and $B''$ are finite constants.
Choosing $\eta\leq (B'')^{-k}\varepsilon$ makes our sum smaller than $\varepsilon$ for all choices of $n$.
\end{proof}
\bigskip

\subsubsection{Proof of Proposition \ref{prop:cvmes} and Corollary \ref{coro:cvmes}}
\label{sec:proofcv}

\noindent \textit{Proof of Proposition \ref{prop:cvmes}.} First note that since $(1-s_1)\nu_k(\mathrm d \mathbf s)$ is a finite measure on $\mathcal S_k$ and since $\nu_k( \exists i:s_i=0)=0$, for all $\varepsilon>0$ there exists a $\eta>0$  such that $\int_{\mathcal S_k} (1-s_1) \mathbbm 1_{\{\exists i :s_i <\eta\}}\nu_k(\mathrm d \mathbf{s}) < \varepsilon$. Together with Lemma \ref{lem:tight}, this implies that Proposition \ref{prop:cvmes} will be proved once we have checked that
\[n^{\frac{1}{k}}\int_{\s_k} (1-s_1)f(\mathbf{s})\prod_{i=1}^k\mathbbm 1_{\{s_i \geq \eta\}} \bar{q}_n(\mathrm d \mathbf{s}) \underset{n \rightarrow \infty} \longrightarrow \int_{\s_k}(1-s_1)f(\mathbf{s})\prod_{i=1}^k\mathbbm 1_{\{s_i \geq \eta\}}\nu_k(\mathrm d\mathbf{s})                            
\]
for all $\eta>0$ and all continuous functions $f$ on $\s_k$ . In the following, we fix such a real number $\eta>0$ and a function $f$. Using the expression (\ref{qn}) and Lemmas \ref{lem:gamma} and \ref{lem:beta}, we see that 
\begin{eqnarray*}
n^{\frac{1}{k}}\int_{\s^k} (1-s_1)f(\mathbf{s})\prod_{i=1}^k\mathbbm 1_{\{s_i \geq \eta\}} \bar{q}_n(\mathrm d \mathbf{s}) \underset{n \rightarrow \infty}\sim \frac{n^{1-k}}{k \Gamma(\frac{1}{k})^{k-1}}  \sum_{\lambda \in \mathcal C_n^k} f\left( \frac{\lambda}{n}\right) \prod_{i=1}^k \left(\frac{\lambda_i}{n}\right)^{-(1-1/k)}\mathbbm{1}_{\{\lambda_i \geq \eta n\}}.
\end{eqnarray*}
We conclude by noticing that this last term is  in fact a Riemann sum of a (Riemann) integrable function on $[0,1]^{k-1}$: to sum over $\lambda\in\mathcal C^k_n$, we only need to choose $\lambda_1,\ldots,\lambda_{n-1}$ in $\{0,...,n\}$ such that $n-(\lambda_1 +\ldots+\lambda_{n-1})\geq0$. Standard results on Riemann sums then imply that it converges towards  the integral $$\int_{\mathcal S_k} (1-s_1)f (\mathbf s)\prod_{i=1}^k \mathbbm{1}_{\{s_i \geq \eta\}} \nu_k(\mathrm d \mathbf s).$$
The convergence of the decreasing versions of the measures follows immediately. A continuous function $f$ on $\mathcal S_k$ being fixed, we let $g_f$ be the function  defined on $\s_k$ by $g_f(\mathbf{s})=(1-s_1^{\downarrow})f(\mathbf{s}^{\downarrow})/(1-s_1)$. The function $g_f$ is then continuous and bounded on $\s_k$ (there is no singularity when $s_1=1$ since $s_1^{\downarrow}=s_1$ as soon as $s_1\geq 1/2$). By the first part of this proof, we then have 
\begin{eqnarray*}
n^{\frac{1}{k}} \int_{\s_k}(1-s_1)f(\mathbf{s})\bar{q}_n^{\downarrow}(\mathrm d \mathbf{s}) &=& n^{\frac{1}{k}} \int_{\s_k} (1-s_1^{\downarrow})f(\mathbf{s}^{\downarrow})\bar{q}_n(\mathrm d \mathbf{s}) =
	n^{\frac{1}{k}} \int_{\s_k} (1-s_1)g_f(\mathbf{s})d\bar{q}_n(\mathbf{s}) \\
	&\underset{n\to\infty}\rightarrow&  \int_{\s_k} (1-s_1)g_f(\mathbf{s})\nu_k(\mathrm d \mathbf{s}) = \int_{\s_k}(1-s_1)f(\mathbf{s})\nu_k^{\downarrow}(\mathrm d\mathbf{s}).
\end{eqnarray*}
$\hfill \square$

\bigskip

\noindent \textit{Proof of Corollary \ref{coro:cvmes}}. Let $f$ be a  continuous function on $\mathcal S_{k',\leq}$ and assume for the moment that $k' \leq k-2$.
Applying first Proposition \ref{prop:cvmes} and then the identity (\ref{beta}), we get
\begin{align*}
& n^{1/k} \int_{\mathcal S_{k',\leq}} f(\mathbf s) (1-s_1)\overline q_n'(\mathrm d \mathbf s) \underset{n \rightarrow \infty} \longrightarrow  \frac{1}{k(\Gamma(\frac{1}{k}))^{k-1}} \int_{\mathcal S_k} f(s_1,...,s_k') \prod_{i=1}^k s_i^{-(1-1/k)}\mathrm d \mathbf s \\
& = \frac{1}{k(\Gamma(\frac{1}{k}))^{k-1}}\int_{S_{k',\leq}} f(\mathbf s) \prod_{i=1}^{k'} s_i^{-(1-1/k)}\left(\int_{\big(0,1-\sum_{i=1}^{k'}s_i\big]^{k-1-k'}}  \prod_{i=k'+1}^k s_i^{-(1-1/k)} \mathrm d s_{k'+1} \dots \mathrm d s_{k-1}\right) \mathrm d \mathbf s \\
& = \frac{1}{k(\Gamma(\frac{1}{k}))^{k-1}}\int_{S_{k',\leq}} f(\mathbf s) \prod_{i=1}^{k'} s_i^{-(1-1/k)}\left(\frac{\Gamma(\frac{1}{k})^{k-k'}}{\Gamma\left(1-k'/k\right)} \left(1-\sum_{i=1}^{k'}s_i\right)^{-k'/k} \right) \mathrm d \mathbf s, 
\end{align*}
which gives the result for $k' \leq k-2$. For $k'=k-1$ the calculation is more direct since we do not need (\ref{beta}). Finally, the convergence of decreasing measures follows immediately by mimicking the end of the proof of Proposition \ref{prop:cvmes}.
$\hfill \square$

%%%%%%%%%%%%%%%%%%%%%%%%%%%%%%%%%%%%%%%%%
\subsection{Markov branching property and identification of the limit}
\label{sec:conclusion}
%%%%%%%%%%%%%%%%%%%%%%%%%%%%%%%%%%%%%%%%%

\begin{prop}[Markov branching property] 
\label{propMB}
Let $n\in\mathbb Z_+$. Conditionally on $(X^i_n)_{i\leq k}$, the $(T^i_n)_{i\leq k}$ are mutually independent and, for $i\leq k$, $T^i_n$ has the same law as $T_{X^i_n}$.
\end{prop}

\begin{proof} We prove this statement by induction on $n\in\mathbb Z_+$. Starting with $n=0$, we have $X^i_0=0$ and $T^i_0=T_0$ for all $i$, everything is deterministic.

Assume now that the Markov branching property has been proven up until some integer $n-1$, and let us prove it for $n$. Let $e$ be the random selected edge of $T_{n-1}$ used to build $T_{n}$ and let $J$ be the random variable defined by: $J=j$ if $e$ is an edge of $T^j_{n-1}$, $j\leq k$, and $J=0$ if $e$ is the edge adjacent to the root of $T_n$. Note that $J$  and $T_n$ are independent conditionally  on $(X^i_{n-1})_{i\leq k}$. Let us then determine the law of $(T^i_n,i\leq k)$ conditionally on $J$ and $(X^i_n)_{i\leq k}$. 

\smallskip

If $J=j\neq0$ then $(T^i_{n})_{i\leq k,i\neq j}$ is the same sequence as $(T^i_{n-1})_{i\leq k,i\neq j}$ and we have added an extra edge to $T^j_{n-1}$. Hence,
for all $j \leq k$ and $(x_1,...,x_k) \in \mathcal C_n^k$, with $x_j \geq 1$, we have for all $k$-uplet of rooted $k$-ary trees $(t_1,...,t_k)$ with respectively $x_1,...,x_k$ internal nodes,
\begin{eqnarray*}
&&\mathbb P\left(T_{n}^i=t_i,  1 \leq i \leq k \  | \ X_{n}^i=x_i, 1 \leq i \leq k, J=j\right) \\
&=& \mathbb P\left(T_{n-1}^i=t_i,  1 \leq i \leq k, i \neq j, T_{n}^j=t_j \  | \ X_{n-1}^i=x_i, 1 \leq i \leq k, i \neq j, J=j\right) \\
&=& \mathbb P\left(T_{n-1}^i=t_i,  1 \leq i \leq k, i \neq j \  | \ X_{n-1}^i=x_i, 1 \leq i \leq k, i \neq j, J=j\right) \mathbb P\left(T_{x_j}=t_j\right) \\
&=& \textstyle \prod_{i=1}^k \mathbb P\left(T_{x_i}^i=t_i \right)
\end{eqnarray*}
where we have used that a conditioned uniform variable is uniform in the set of conditioning to get the second equality  and then that $J$  and $T_n$ are independent conditionally  on $(X^i_{n-1})_{i\leq k}$, together with the Markov branching property at $n-1$, to get the third equality.

When $J=0$, $(T^i_{n})_{i\leq k}=(T_n,T_0,\ldots,T_0)$. Since $T_0$ is deterministic and the event $\{J=0\}$ is independent of $T_n$, the distribution of $(T^i_{n})_{i\leq k}$ conditional on $J=0$ and $(X^i_n)_{i\leq k}=(n,0,...,0)$ is indeed the same as that of the $k$-uplet of independent random variables $(T_n,T_0,\ldots,T_0)$.

Finally, since the distribution of $(T^i_{n})_{i\leq k}$ conditionally on $J$ and $(X^i_{n})_{i\leq k}$ is independent of $J$, one can remove $J$ in the conditioning, which ends the proof.
\end{proof}

\bigskip

We now have the material to prove the convergence in distribution of $n^{-1/k}T_n$ and identify its limit as a fragmentation tree.

\bigskip

\noindent \textbf{Proof of Theorem \ref{thmain} (convergence in distribution part).} Theorem 5 of \cite{HM12} concerns  sequences of Markov branching trees indexed by their number of leaves, however our sequence $(T_n)$ is indexed by the number of internal nodes of the tree. This is not a real problem since $T_n$ has $1+(k-1)n$ leaves for all $n$,  the sequence $(T_n)_{n \in \mathbb N}$ can be seen as a sequence of Markov branching trees $(T_p^{\circ})_{p\in1+(k-1)\mathbb N}$ indexed by their number of leaves. 
For all $p\in1+(k-1)\mathbb N$, we let $\bar q_p^{\circ}$ denote its associated splitting distribution, that is, if $p=1+(k-1)n$, $\bar q_p^{\circ}$ is the distribution on  $\mathcal S_k$ of the sequence 
$$
\left(\frac{1+(k-1)X_{n-1}^i}{1+(k-1)n} \right)_{i \leq k}.
$$
As an immediate consequence of Proposition \ref{prop:cvmes}, we have that
\begin{equation}
\label{cvqncirc}
(k-1)^{-1/k}p^{1/k} (1-s_1)\bar{q}_p^{\circ,\downarrow}(\mathrm d \mathbf{s})\underset{\underset{p \in 1+(k-1)\mathbb N}{p \rightarrow \infty}}\Rightarrow (1-s_1)\nu_k^{\downarrow}(\mathrm d \mathbf{s}).
\end{equation}
Indeed, for any bounded Lipschitz function $f:\mathcal S_k \rightarrow \mathbb R$,  let $g_f:\mathcal S_k \rightarrow \mathbb R$ be defined by $g_f(\mathbf s)=(1-s_1)f(\mathbf s)$. Then $g_f$  is also Lipschitz, say with Lipschitz constant $c_g$. It is then easy to see that
$$
n^{1/k} \left| \mathbb E \Bigg[ g\Bigg( \left(\frac{1+(k-1) X^i_{n-1}}{1+(k-1)n}\right)^{\downarrow}\Bigg) \Bigg] - \mathbb E \Bigg[g\Bigg(\Bigg(\frac{X^i_{n-1}}{n-1} \Bigg)^{\downarrow}\Bigg) \Bigg]\right| \leq \ \frac{ n^{1/k}2k c_g}{1+(k-1)n} \ \underset{n \rightarrow \infty} \rightarrow 0.
$$
Together with  Proposition \ref{prop:cvmes} this immediately leads to (\ref{cvqncirc}).

Hence the sequence $(T_p^{\circ})_{p\in1+(k-1)\mathbb N}$ is Markov branching with a splitting distribution sequence $(\bar q_p^{\circ})$ satisfying (\ref{cvqncirc}). This is exactly the hypotheses we need to apply Theorem 5 of \cite{HM12}, except that this theorem is stated for sequences of Markov branching  trees indexed by the full set $\mathbb N$, not by one of its subsets. However, without any modifications, it could easily be adapted to that setting. Hence we obtain from this theorem that
$$
\left((k-1)^{1/k}p^{-1/k}T^{\circ}_p, \mu_p^{\circ}\right) \underset{\underset{p \in 1+(k-1)\mathbb N}{p \rightarrow \infty}}\longrightarrow (\mathcal T_k,\mu_k)
$$ 
where $\mu_p^{\circ}$ is the uniform probability on the leaves of $(T_p^{\circ})$ and $(\mathcal T_k,\mu_k)$ the fragmentation tree of Theorem \ref{thmain}.
This convergence holds in distribution, for the GHP topology. Otherwise said, $(n^{-1/k}T_n)$ endowed with the uniform probability on its leaves converges in distribution towards $(\mathcal T_k,\mu_k)$.
$\hfill \square$

\bigskip

%%%%%%%%%%%%%%%%%%%%%%%%%%%%%%%%%%%%%%%%%%%%%%%%%%%%%%%
\section{Convergence in probability and joint convergence}
\label{sec:cvproba}
%%%%%%%%%%%%%%%%%%%%%%%%%%%%%%%%%%%%%%%%%%%%%%%%%%%%%%%
This section is dedicated to improving the convergence in distribution we have just obtained. We will construct the limiting tree in the space $\ell^1$ of summable real-valued sequences (equipped with its usual metric $d_{\ell^1}$), and convergence will be proved by using subtrees akin to finite-dimensional marginals. The almost sure convergence of these marginals can be proved using urn schemes and results concerning Chinese restaurant processes, as studied by Pitman in \cite[Chapter 3]{PitmanStFl}. Tightness properties will extend this to the convergence of $(n^{-1/k}T_n,\mu_n)$. Unfortunately, almost sure convergence is lost by this method and we are left with convergence in probability. Also, due to some technical issues, we first have to study the Gromov-Hausdorff convergence of the non-measured trees before adding the measures.

%%%%%%%%%%%%%%%%%%%%%%%%%%%%%%%%%%
\subsection{Finite-dimensional marginals and the limiting tree}
\label{sec:marginals}
%%%%%%%%%%%%%%%%%%%%%%%%%%%%%%%%%%

In this section we will need to define an ordering of the leaves of $T_n$ for $n\in\Z_+$, calling them $(L_n^i)_{1\leq i\leq (k-1)n+1}$. They are labelled by order of apparition: the single leaf of $T_0$ is called $L_0^1$, while, given $T_n$ and its leaves, the leaves $L_{n+1}^{1},\ldots,L_{n+1}^{(k-1)n+1}$ of $T_{n+1}$ are those inherited from $T_n$, and the leaves $L_{n+1}^{(k-1)n+2}, \ldots,L_{n+1}^{(k-1)n+k}$ are the leaves at the ends of the new edges labelled $2,3,\ldots,k$ respectively.

Let $p\in\Z_+$. For all $n\geq p$, consider the subtree $T^p_n$ of $T_n$ spanned by the root and all the leaves $L^i_n$ with $i\leq (k-1)p+1$. The tree $T^p_n$ has the same graph structure as $T_p$, however the metric structure isn't the same: the distance between two vertices of $T^p_n$ is the same as the distance between the corresponding vertices of $T_n$. The study of the sequence $(T^p_n)_{n\geq p}$ for all $p$ will give us much information on the sequence $(T_n,\mu_n)_{n\in\Z_+}$.

\begin{prop}
\label{prop:cvmarginalesGH}
Let $p\in\Z_+$. We have, in the Gromov-Hausdorff sense, as $n$ goes to infinity:
\begin{equation}
\frac{T^p_n}{n^{1/k}}  \overset{\mathrm{a.s.}}\longrightarrow \T^p,
\end{equation}
where $\T^p$ is a rooted compact $\R$-tree with $(k-1)p+1$ leaves which we will call $(L^i)_{i\leq (k-1)p+1}$. Under a suitable embedding in $\ell^1$, for $p'<p$, $\T^{p'}$ is none other than the subtree of $\T^p$ spanned by the root and the leaves $L^i$ for $i\leq (k-1)p'+1$, making this notation unambiguous.
\end{prop}

\begin{proof} The proof hinges on our earlier description of $T_n^p$ for $n\geq p$: it is the graph $T_p$, but with distances inherited from $T_n$. As explained in Lemma \ref{lem:cvGHarbresdiscrets}, we only need to show that, for $i$ and $j$ smaller than $(k-1)p+1$, both $n^{-1/k} d(L^i_n,L^j_n)$ and  $n^{-1/k} d(\rho,L^i_n)$ have finite limits as $n$ goes to infinity. We first concentrate on the case of $n^{-1/k}d(\rho,L^1_n)$. This could be done by noticing that $(d(\rho,L^1_n))_{n \geq 0}$ is a Markov chain and using martingale methods, however, in view of what will follow, we will use the theory of Chinese restaurant processes.

For $n\in\N$, we consider a set of tables indexed by the vertices of $T_n$ which are strictly between $\rho$ and $L^1_n$. We then let the number of clients on the table indexed by a vertex $v$ be the number of internal nodes $u$ of $T_n$ such that $v$ is the branch point of $u$ and $L_1^n$ (including the case $u=v$).

\begin{figure}
\centering

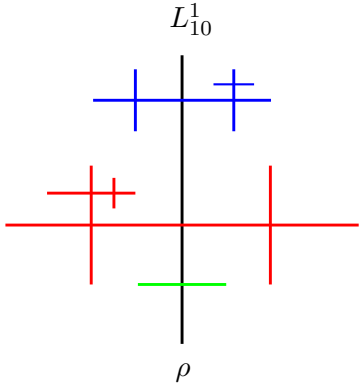
\caption{Colour-coding of the tables of $T_{10}$ (here $k=3$). The green table has one client, the red table has five and the blue table has four.}
\end{figure}

Let us check that this process is part of the two-parameter family introduced by Pitman in \cite{PitmanStFl}, Chapter 3, with parameters $(1/k,1/k)$. Indeed, assume that, at time $n\in\N$, we have $l\in\N$ tables with respectively $n_1,\ldots,n_l$ clients (the tables can be ordered by their order of apparition in the construction). For any $i\leq l$, table $i$ corresponds to a subset of $T_n$ with $kn_i-1$ edges, thus there is a probability of $(kn_i-1)/(kn+1)$ that the next client comes to this table. This next client will sit a a new table if the selected edge is between $\rho$ and $L^1_n$, an event with probability $(l+1)/(kn+1)$.

Since, for all $n$, $d(\rho,L^1_n)$ is equal to the number of tables plus one, Theorem 3.8 of \cite{PitmanStFl} tells us that $n^{-1/k}d(\rho,L^1_n)$ converges almost surely towards a $(1/k,1/k)$-generalized Mittag-Leffler random variable (the definition of generalized Mittag-Leffler distributions is recalled in Section \ref{sec:ML}, however we will not need here the exact distribution of this limit). The cases of $d(\rho,L^i_n)$ and $d(L^i_n,L^j_n)$ for $i\neq j$ can be treated very much the same way: the main difference is that the tables of the restaurant process are now indexed by the nodes between $L^i_n$ and $L^j_n$, and they have a non-trivial initial configuration. Lemma \ref{lem:cvGHarbresdiscrets} finally implies that $n^{-1/k}T_n^p$ does converge a.s. to a tree with $(k-1)p+1$ leaves in the Gromov-Hausdorff sense.

The trees $(\mathcal T^p, p\in \mathbb Z_+)$ can be embedded in $\ell^1$ as a growing sequence of trees using the so-called stick-breaking method of \cite{Ald93}, Section 2.2, by sequentially adding each leaf. Let us explain this in reasonable detail. First, we embed the root $\rho$ as the null vector $\big(0,0,\ldots)$ and the first leaf $L^1$ is as $\big(ht(L^1),0,\ldots\big)$. Once the leaves $L^1,\ldots,L^{i-1}$, with $i \geq 2$, have been embedded in $\ell^1$, in order to add the $i$-th leaf $L^i$, first locate the point $H^i$ where the path $[[\rho,L^i]]$ splits off from the subtree containing the root and all the leaves $L^j$ with $j<i$. This point is of course already embedded in $\ell^1$, so we can continue the embedding by adding from that point a line segment following the $i$-th coordinate of length $d(H^i,L^i)$, which completes the embedding of $[[\rho,L^i]]$.
\end{proof}

Under this embedding in  $\ell^1$ we let
	\[\T=\overline{\cup_{p=0}^{\infty}\T^p}, 
\]
which is also an $\R$-tree. We will see in Lemma \ref{lem:cvdesarbreslimites} and Proposition \ref{prop:cvprobaGH} that this tree is compact and is the limiting tree for $n^{-1/k} T_n$, the tree which was called $\T_k$ in the introduction.

%%%%%%%%%%%%%%%%%%%%%%%%%%%%%%%%%%
\subsection{A tightness property}
\label{sec:tightness}
%%%%%%%%%%%%%%%%%%%%%%%%%%%%%%%%%%
To move from the convergence of $T^p_n$ for all $p\in\N$ to the convergence of $T_n$, we need some kind of compactness to not be bothered by the choice of $p$, which the following proposition gives.

\begin{prop}
\label{prop:tightness}
For all $\varepsilon>0$ and $\eta>0$, there exists an integer $p$ such that, for $n$ large enough,
	\[\mathbb{P}\big(d_{\mathrm{GH}}(T^p_n,T_n)>n^{\frac{1}{k}}\eta\big) <\varepsilon.
\]
The same is then true if we replace $p$ by any greater integer $p'$.
\end{prop}

Before proving this proposition, we need an intermediate result. Fix $p\in\Z_+$. All the variables in the following lemma depend on a variable $n\geq p$, however we omit mentioning $n$ for the sake of readability.
\begin{lemma}
\label{lem:pseudoMB}
Let $v_1,\ldots,v_N$ be the internal nodes of $T_n$ which are part of $T^p_n$ but are not branch points of $T^p_n$, listed in order of apparition. At each of these vertices are rooted $k-1$ subtrees of $T_n$ which we call $(S^i_j\, ; \, {j\leq N,i\leq k-1})$, $S^i_j$ being the tree rooted at $v_j$ with a unique edge adjacent to $v_j$, this edge having label $i+1$. Letting, for $j\leq N$ and $i\leq k-1$, $Y^i_j$ be the number of internal nodes of $S^i_j$ then, conditionally on $(Y^l_q; q\leq N, l\leq k-1)$, the tree $S^i_j$ has the same distribution as $T_{Y^i_j}$.

Furthermore, these subtrees allow us to define some restaurant processes by letting $n$ vary: for $j\leq N$, let $S_j=\cup_{i=1}^{k-1}S^i_j$, and $Y_j$ be the number of vertices of $S_j$, including $v_j$ but excluding any leaves. Considering $S_j$ as a table with $Y_j$ clients for all $j$, we have defined a restaurant process whose initial configuration is zero tables at time $n=p$ and has parameters $(1/k,p+1/k)$.
\end{lemma}

The subtrees $(S^i_j)$ are also conditionally independent, however this will not be useful to us.
\begin{figure}[ht]
\centering

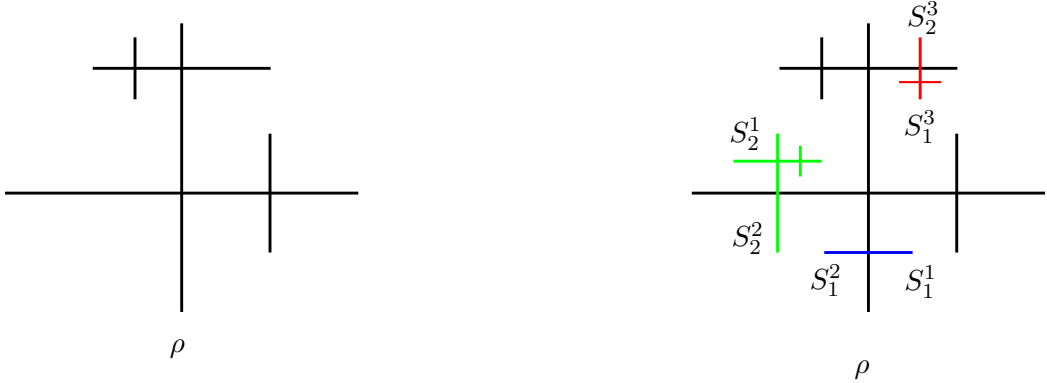
\caption{The tree $T_{10}$ seen as an extension of $T_4$ ($k=3$). The colored sections correspond to the tables of the Chinese restaurant, each table corresponding to two subtrees.}
\end{figure}

\begin{proof} The proof that $S^i_j$ is, conditionally on  $(Y^l_q; q\leq N, l\leq k-1)$, distributed as $T_{Y^i_j}$ is a straightforward induction on $n$. We will not give details since this is very similar to the Markov branching property (Proposition \ref{propMB}) but the main point is that, conditionally on the event that the selected edge at a step of the algorithm is an edge of $S^i_j$, this edge is then uniform amongst the edges of $S^i_j$.

The restaurant process nature of these subtrees is proved just as in Proposition \ref{prop:cvmarginalesGH}: if table $S_j$ has $Y_j$ clients at time $n \geq p$, then the subtree $S_j$ has $kY_j-1$ edges, and a new client will therefore be added to this table with probability $(kY_j-1)/(kn+1)$, while a new table is formed with probability $(kp+N+1)/(kn+1)$. These are indeed the transition probabilities of a restaurant process with parameters $(1/k,p+1/k)$ taken at time $n-p$.
\end{proof}

\noindent \textbf{Proof of Proposition \ref{prop:tightness}.} We will need Lemma 33 of \cite{HM12}: since the sequence $(T_n)_{n\in\Z_+}$ is Markov branching and we have the convergence of measures of Proposition \ref{prop:cvmes}, we obtain that, for any $q>0$, there exists a finite constant $C_q$ such that, for any $\varepsilon>0$ and $n\in\N$,
	\[\mathbb{P}\Big(ht(T_n)\geq \varepsilon n^{1/k}\Big)\leq \frac{C_q}{\varepsilon^q}.
\]
Choosing $q>k$, applying this to all of the $S_j^i$ conditionally on $Y_j^i$ ($j\leq N,i\leq k-1$), and using the simple fact that 
	\[d_{\mathrm{GH}}(T^p_n,T_n)\leq \underset{i,j}\max\; ht(S^j_i),
\]
we obtain
\begin{align*}
\mathbb{P}\Big(d_{\mathrm{GH}}(T^p_n,T_n)>\eta n^{1/k} \; |\; (Y_j^i)_{i,j}\Big) &\leq \sum_{i,j} \mathbb{P}\Big(ht(S^i_j)>\eta n^{\frac{1}{k}}\;|\; (Y^i_j)_{i,j}\Big) \\
                                                &\leq \sum_{i,j} \mathbb{P} \left(ht(S_j^i)>\eta \Big(\frac{n}{Y^i_j}\Big)^{\frac{1}{k}}(Y^i_j)^{\frac{1}{k}}\;|\; (Y^i_j)_{i,j}\right) \\
                                                &\leq \frac{C_q}{\eta^q} \sum_{i,j} \left(\frac{Y^i_j}{n}\right)^{\frac{q}{k}} \\
                                                &\leq \frac{C_q}{\eta^q} \sum_j \left(\frac{Y_j}{n}\right)^{\frac{q}{k}}.
\end{align*}
Let us now reorder the $(Y_j)$ in decreasing order. Theorem 3.2 of \cite{PitmanStFl} states the following convergence for all $j$ as $n$ goes to infinity:
	\[\frac{Y_j}{n} \overset{\mathrm{a.s.}}\longrightarrow V_j
\]
where $(V_j)_{j\in\N}$ is a Poisson-Dirichlet random variable with parameters $(1/k,p+1/k)$. By writing out, for each $j$, $(Y_j)^{\frac{q}{k}}\leq (Y_1)^{\frac{q}{k}-1}Y_j$, we then get
	\[\underset{n\to\infty}\limsup \; \mathbb{P}\Big(d_{\mathrm{GH}}(T^p_n,T_n)>\eta n^{1/k}\Big) \leq \frac{C_q}{\eta^q} \mathbb{E}[(V_1)^{\frac{q}{k}-1}].
\]
We then use an estimation of the density of $V_1$ found in Proposition 19 of \cite{PitmanYor} to obtain
\begin{align*}
\mathbb{E}[(V_1)^{\frac{q}{k}-1}] &\leq \frac{\Gamma (p+1+\frac{1}{k})}{\Gamma(p+\frac{2}{k})\Gamma(1-\frac{1}{k})} \int_0^1 u^{\frac{q-1}{k}-2}(1-u)^{\frac{2}{k}+p-1} du \\
                            &\leq\frac{\Gamma (p+1+\frac{1}{k})\Gamma(\frac{q-1}{k}-1)\Gamma(p+\frac{2}{k})}{\Gamma(p+\frac{2}{k})\Gamma(1-\frac{1}{k})\Gamma(p-1+\frac{q+1}{k})}.
\end{align*}
As $p$ goes to infinity, this is, up to a constant, equivalent to $p^{2-\frac{q}{k}}$, which tends to $0$ if we take $q>2k$, thus ending the proof.
\qed

%%%%%%%%%%%%%%%%%%%%%%%%%%%%%%%%%%
\subsection{Gromov-Hausdorff convergence}
\label{sec:GH}
%%%%%%%%%%%%%%%%%%%%%%%%%%%%%%%%%%
In this section and the next two sections we work with the versions of the trees $\mathcal T^p$ simulnaeously embedded in $\ell^1$ (see Proposition \ref{prop:cvmarginalesGH}) and we recall that $\T=\overline{\cup_{p=0}^{\infty}\T^p}$.

\begin{lemma}\label{lem:cvdesarbreslimites} As $p$ tends to infinity, we have the following convergence, in the sense of Hausdorff convergence for compact subsets of $\ell^1$:
	\[\T^p \overset{a.s.}\longrightarrow \T.\]
In particular, the tree $\T$ is in fact compact and $\mathcal T^p$ converges a.s. to $\mathcal T$ in the Gromov-Hausdorff sense.
\end{lemma}

\begin{proof} Let us first prove that the sequence $(\T^p)_{p\in\N}$ is Cauchy in probability for the Hausdorff distance in $\ell^1$, in the sense of \cite{Kallenberg}, Chapter 3: we want to show that, for any $\varepsilon>0$ and $\eta>0$, if $p$ and $q$ are large enough, $\mathbb{P}\big(d_{\ell^1,\mathrm{H}}(\T^p,\T^q)>\eta\big)<\varepsilon.$ Let therefore $\eta>0$ and $\varepsilon>0$. We have, for integers $p$ and $q$,
\begin{align*}
\mathbb{P}\left(d_{\ell^1,\mathrm{H}}(\T^p,\T^q)>\eta\right)&= \mathbb{P}\left(\underset{n\to\infty}\lim n^{-1/k}d_{\ell^1,\mathrm{H}}(T^p_n,T^q_n)>\eta  \right) \\
             &\leq \underset{n\to\infty}\liminf \, \mathbb{P}\left(n^{-1/k}d_{\ell^1,\mathrm{H}}(T^p_n,T^q_n)>\eta \right) \\
             &\leq \underset{n\to\infty}\liminf \, \mathbb{P}\left(d_{\ell^1,\mathrm{H}}(T^p_n,T_n)+d_{\ell^1,\mathrm{H}}(T^q_n,T_n) >n^{1/k}\eta\right) \\
             &\leq  \underset{n\to\infty}\limsup \, \mathbb{P}\left(d_{\ell^1,\mathrm{H}}(T^p_n,T_n)>n^{1/k}\frac{\eta}{2}\right) + \underset{n\to\infty}\limsup \, \mathbb{P}\left(d_{\ell^1,\mathrm{H}}(T^q_n,T_n)>n^{1/k}\frac{\eta}{2}\right).
\end{align*}
Thus, by Proposition \ref{prop:tightness}, choosing $p$ and $q$ large enough yields
	\[\mathbb{P}\Big(d_{\ell^1,\mathrm{H}}(\T^p,\T^q)>\eta\Big) \leq \varepsilon.
\]
Since the Hausdorff metric on the set of nonempty compact subsets of $\ell^1$ is complete, the sequence $(\T^p)_{p\in\N}$ does converge in probability, and thus has an a.s. converging subsequence. Since it is also monotonous (in the sense of inclusion of subsets), it in fact does converge to a limit we call $\mathcal{L}$, and we only need to show that $\mathcal{L}=\T$. Since $\mathcal{L}$ is a compact subset of $\ell^1$ and contains $\T^p$ for all $p$, we have $\T\subset\mathcal{L}$. On the other hand, assuming the existence of a point $x\in\mathcal{L}\setminus\T$ would yield $\varepsilon>0$ such that $d_{\ell^1}(x,\T)\geq\varepsilon$ and also $d_{\ell^1}(x,\T^p)\geq\varepsilon$ for all $p$, negating the Hausdorff convergence of $\T^p$ to $\mathcal{L}$.
\end{proof}

\begin{prop}
\label{prop:cvprobaGH}
We have
	\[\frac{T_n}{n^{1/k}} \overset{\mathbb{P}}\longrightarrow \T
\]
as $n$ goes to infinity, in the Gromov-Hausdorff sense.
\end{prop}

\begin{proof}
All the work has already been done, we only need to stick the pieces together. Let $n\in\N$ and $p\leq n$, we use the triangle inequality:
	\[d_{\mathrm{GH}}\bigg(\frac{T_n}{n^{1/k}},\T\bigg)\leq d_{\mathrm{GH}}\bigg(\frac{T_n}{n^{1/k}},\frac{T_n^p}{n^{1/k}}\bigg) + d_{\mathrm{GH}}\bigg(\frac{T_n^p}{n^{1/k}},\mathcal T^p\bigg) + d_{\mathrm{GH}}(\T^p,\T).
\]
For $\eta>0$, we then have
\begin{eqnarray*}
&& \mathbb{P}\Bigg(d_{\mathrm{GH}} \bigg(\frac{T_n}{n^{1/k}},\T\bigg)>\eta\Bigg) \\
&\leq& 
	      \mathbb{P}\Bigg(d_{\mathrm{GH}}\bigg(\frac{T_n}{n^{1/k}},\frac{T_n^p}{n^{1/k}}\bigg)>\frac{\eta}{3}\Bigg) +
	      \mathbb{P}\Bigg(d_{\mathrm{GH}}\bigg(\frac{T_n^p}{n^{1/k}},\mathcal T^p\bigg)>\frac{\eta}{3}\Bigg) +
	      \mathbb{P}\Bigg(d_{\mathrm{GH}}(\T^p,\T)>\frac{\eta}{3}\Bigg).	      
\end{eqnarray*}
Let $\varepsilon>0$. By Lemma \ref{lem:cvdesarbreslimites} and Proposition \ref{prop:tightness}, there exists $p$ such that the third term of the sum is smaller than $\varepsilon$, and the first term also is for all $n$ large enough. Apply then Proposition \ref{prop:cvmarginalesGH} with this fixed $p$, to make the second term smaller than $\varepsilon$ for large $n$, and the proof is over.
\end{proof}

%%%%%%%%%%%%%%%%%%%%%%%%%%%%%%%%%%
\subsection{Adding in the measures}
\label{sec:GHP}
%%%%%%%%%%%%%%%%%%%%%%%%%%%%%%%%%%
We now know that $\T$ is compact. This compactness will enable us to properly obtain a measure on $\T$ and the desired GHP convergence. For all $n$ and $p\leq n$, let $\mu_n^p$ be the image of $\mu_n$ by the projection from $T_n$ to $T_n^p$ (see Appendix \ref{sec:A2} for a precise definition of projection).  Let also, for all $p$, $\pi^p$ be the projection from $\T$ to $\T^p$.
We start by proving an extension of Proposition \ref{prop:cvmarginalesGH} to the measured case.

\begin{prop}
\label{prop:cvmarginalesGHP}
There exists a probability measure $\mu^p$ on $\T^p$ such that, in the GHP sense,
	\[\Big(\frac{T^p_n}{n^{1/k}},\mu^p_n\Big) \overset{a.s.}\to (\T^p,\mu^p).
\]
What's more, we have, for $p'\geq p$, $\mu^p=(\pi^p)_{*}\mu^{p'}$.
\end{prop}

\begin{proof} We aim to apply Lemma \ref{lem:cvGHParbresdiscrets}. For this we first embed the trees $T_n^p$ and $\mathcal T^p$ in $\ell^1$ with the stick-breaking method, by sequentially adding the leaves according to their indices, as recalled at the end of the proof of Proposition \ref{prop:cvmarginalesGH}. 

The first step to apply Lemma \ref{lem:cvGHParbresdiscrets} is then to find an appropriate dense subset of $\mathcal T^p$.
Since we know from Section \ref{sec:cvd} that the distribution of the metric space $\T$ is that of a fragmentation tree and that the dislocation measure $\nu_k$ has infinite total mass, Theorem 1 from \cite{HM04} tells us that it is leaf-dense. As a consequence, its branch points are also dense. Let $S_p$ be the set of points of $\T^p$ which are also branch points of $\T$, we then know that $S_p$ is a dense subset of $\T^p$. In fact $S_p$ can be simply explicited:
	\[S_p = \{L^i\wedge L^j;\: i\leq (k-1)p+1 \text{ or }  j\leq (k-1)p+1\}
\]
(recall that  $\{L^i, i \geq 1\}$ is the set of leaves of $\mathcal T$ that belong to $\cup_{p=0}^{\infty} \mathcal T^p$).
Let $i$ and $j$ be integers such that either $i$ or $j$ is smaller than or equal to $(k-1)p+1$, and let $x=L^i\wedge L^j$. For $n$ such that $i\leq (k-1)n+1$ and $j\leq (k-1)n+1$, define $x_n$ as the branch point in $T_n$ of $L^i_n$ and $L^j_n$. It is immediate that $x_n$ converges to $x$, and moreover, calling $(T^p_n)_{x_n}$ the subtree of descendants of $x_n$ in $(T^p_n)$, that $(T^p_n)_{x_n}$ converges to $\T^p_x$ (the subtree of descendants of $x$ in $\T^p$) in the Hausdorff sense in $\ell^1$. What is left for us to do is to prove that $\mu^p_n\big(\big(T^p_n\big)_{x_n}\big)=\mu_n\big((T_n)_{x_n}\big)$ converges a.s. as $n$ goes to infinity. To this effect, we let $Z_n$ be the number of internal nodes of $(T_n)_{x_n}$, including $x_n$ itself. Since we have
	\[\mu_n\big((T_n)_{x_n}\big)= \frac{(k-1)Z_n +1}{(k-1)n+1},
\]
convergence of $\mu_n\big((T_n)_{x_n}\big)$ as $n$ goes to infinity is equivalent to convergence of $n^{-1}Z_n$. However the distribution of $Z_n$ is governed by a simple recursion: for all $n$, given $Z_n$, $Z_{n+1}=Z_n+1$ with probability $(kZ_n)/(1+kn)$, while $Z_{n+1}=Z_n$ with the complementary probability. It is then easy to check that the rescaled process $(Z_n/(kn+1))_{n \in \mathbb N}$ is a non-negative martingale, hence converges a.s.. Then, so does $\mu_n\big((T_n)_{x_n}\big)=n^{-1}Z_n$. Hence we can apply Lemma \ref{lem:cvGHParbresdiscrets} to conclude.

The fact that $\mu^p=(\pi^p)_{*}\mu^{p+1}$ is then a direct consequence of the fact that $\mu^p_n=(\pi^p_n)_{*}\mu^{p+1}_n$ for all $n$: for any $x$ in $S_p$, we have $\mu^p_n\big(\big(T^p_n\big)_{x_n}\big)=\mu^{p+1}_n\big(\big(T^{p+1}_n\big)_{x_n}\big)$ and, letting $n$ tend to infinity (and taking left-continuous versions in $x$ as stated in Lemma \ref{lem:cvGHParbresdiscrets}), we obtain $\mu^p\big((\T^p)_x\big)=\mu^{p+1}\big((\T^{p+1})_x\big)$, and Lemma \ref{lem:projmesure} ends the proof.
\end{proof}

\begin{lemma} 
\label{lem:constrmu}
As $p$ tends to infinity, $\mu^p$ converges a.s. to a probability measure $\mu$ on $\T$ which satisfies, for all $p$, $\mu^p=(\pi^p)_*\mu$
\end{lemma}
\begin{proof} Since $\T$ is compact, Lemma \ref{lem:mesuresarbres} shows that we can define a unique measure $\mu$ on $\T$ such that, for all $p$ and $x\in\T^p$, $\mu(\T_x)=\mu^p(\T^p_x)$ (Proposition \ref{prop:cvmarginalesGHP} assures us that this is well-defined since it does not depend on the choice of $p$). By definition, we then have $\mu^p=(\pi^p)_*\mu$ for all $p$, and Lemma \ref{lem:projcontrol} ends the proof.
\end{proof}

\noindent \textbf{Proof of Theorem \ref{thmain} (convergence in probability part).} We want to prove that
\begin{equation}
\label{cvproba}
\Big(\frac{T_n}{n^{1/k}},\mu_n\Big) \overset{\mathbb P}\to (\T,\mu).
\end{equation}
Once this will be done, the distribution of $(\T,\mu)$ will be that of the fragmentation tree mentioned in Theorem \ref{thmain}, since we have already proved the convergence in distribution to that measured tree in Section  \ref{sec:cvd}. 
To get (\ref{cvproba}), notice that Lemma \ref{lem:projcontrol} directly improves Proposition \ref{prop:tightness}, since we can replace the GH distance by the GHP distance, adding the measures $\mu_n$ and $\mu_n^p$ respectively to the trees $T_n$ and $T_n^p$. Once we know this, as well as Proposition \ref{prop:cvmarginalesGHP} and Lemma \ref{lem:constrmu}, the same proof as that of Proposition \ref{prop:cvprobaGH} works.
\qed

%%%%%%%%%%%%%%%%%%%%%%%%%%%%%%%%%%
\subsection{Joint convergence}
\label{sec:joint}
%%%%%%%%%%%%%%%%%%%%%%%%%%%%%%%%%%
For the sake of clarity, we return to the notations of the introduction: for $n\in\Z_+$, $T_n(k)$ is the tree at the $n$-th step of the algorithm, its scaling limit is $\T_k$. For $p\leq n$, we let $T^p_n(k)$ and $\T^p_k$ be the respective finite-dimensional marginals we have studied, endowed, respectively, with the probability measures $\mu_n^p(k)$ and $\mu^p_k$. Let $k'$ be an integer with $2\leq k' < k$. Recall now that $T_n(k,k')$ is the subtree of $T_n(k)$ obtained by discarding all edges with labels greater than or equal to $k'+1$, as well as their descendants. The objective of this section is to prove the convergence in probability of $n^{-1/k}T_n(k,k')$ by using what we know of the convergence of $n^{-1/k}T_n(k)$. This method once again fails to give the distribution of the limiting tree, which will be obtained in Section \ref{sec:loinonconservatif}.

For all $n$, the tree $T_n(k,k')$ comes with a measure $\mu_n(k,k')$ which is the image of $\mu_n(k)$ by the projection from $T_n(k)$ onto $T_n(k,k')$. Similarly, for $p\leq n$, define $$T_n^p(k,k')=T_n(k,k')\cap T_n^p(k),$$ and the image measure $\mu_n^p(k,k')$. For fixed $p$, the almost sure convergence of $n^{-1/k}{T_n^p(k)}$ to $\T^p_k$ as $n$ goes to infinity allows us to extend the edge labellings to $\T^p_k$, and thus define $\T^p_{k,k'}$ and $\mu^p_{k,k'}$ in analogous fashion. Note that the sequence $\big(n^{-1/k}T_n^p(k,k'),\mu_n^p(k,k')\big)$ converges almost surely to $(\T^p_{k,k'},\mu^p_{k,k'})$ as $n$ goes to infinity, by using Lemmas \ref{lem:cvGHarbresdiscrets} and \ref{lem:cvGHParbresdiscrets} and imitating the proofs of Propositions \ref{prop:cvmarginalesGH} and \ref{prop:cvmarginalesGHP}. Finally, considering again versions of all these trees embedded in $\ell^1$ via the stick-breaking construction, we let $$\T_{k,k'}=\overline{\cup_{p=0}^{\infty} \T^p_{k,k'}}.$$ Clearly,  $\T_{k,k'} \subset \T_k$ and we let $\mu_{k,k'}$ be the image of $\mu_{k}$ under the projection from $\T_k$ onto $\T_{k,k'}$.

\medskip

\noindent \textbf{Proof of Theorem \ref{thjoint} (convergence in probability part).} What we want to show is that the sequence of measured trees $\big(n^{-1/k}T_n(k,k'),\mu_n(k,k')\big)$ converges in probability to $(\T_{k,k'},\mu_{k,k'})$ as $n$ goes to infinity, and it is in fact a simple consequence of Lemma \ref{lem:projlip}. Indeed, this lemma directly gives us the fact that, for $p\leq n$,
	\[d_{\mathrm{GHP}}\left(\left(\frac{T_n^p(k,k')}{n^{1/k}},\mu_n^p(k,k')\right),\left(\frac{T_n(k,k')}{n^{1/k}},\mu_n(k,k')\right)\right)\leq d_{\mathrm{GHP}}\left(\left(\frac{T_n^p}{n^{1/k}},\mu_n^p\right),\left(\frac{T_n}{n^{1/k}},\mu_n\right)\right),
\]
as well as, for any $p$,
	\[d_{\mathrm{GHP}}\left(\left(\T^p_{k,k'},\mu^p_{k,k'}\right),(\T_{k,k'},\mu_{k,k'})\right)\leq d_{\mathrm{GHP}}\left((\T^p_k,\mu^p_k),(\T_k,\mu_k)\right).
\]
Since we know that $(\T^p_k,\mu^p_k)\rightarrow(\T_k,\mu_k)$ a.s. as $p \rightarrow \infty$, that $\big(n^{-1/k}T_n^p(k,k'),\mu_n^p(k,k')\big)\rightarrow (\T^p_{k,k'},\mu^p_{k,k'})$ a.s. for all $p$ as $n \rightarrow \infty$, and that there exists a GHP version of Proposition \ref{prop:tightness} (see the convergence in probability part of the proof of Theorem \ref{thmain}), the proof can then be ended just as that of Proposition \ref{prop:cvprobaGH}. $\hfill \square$

%%%%%%%%%%%%%%%%%%%%%%%%%%%%%%%%%%%%%%%%%%%%%%%%%%%%%%%
\section{Stacking the limiting trees}
\label{sec:stack}
%%%%%%%%%%%%%%%%%%%%%%%%%%%%%%%%%%%%%%%%%%%%%%%%%%%%%%%

This section is devoted to the study of $\T_{k,k'}$, seen as a subtree of $\T_k$. We start by giving the distribution of the measured tree $(\T_{k,k'},\mu_{k,k'})$, then move on to prove (\ref{lawsubtree}), which is the last part of Theorem \ref{thjoint}, and then finally show that, even without the construction algorithm, one can extract from $\T_k$ a tree distributed as $\T_{k,k'}$. In Subsections \ref{sec:loinonconservatif} and \ref{sec:extraction} below, some of our arguments rely on specific properties of  fragmentation processes and fragmentation trees. We invite the reader which is not familiar with these topics to refer to \cite{BertoinHomogeneous, BertoinSSF, BertoinBook, HM04, Steph13} for background information.

%%%%%%%%%%%%%%%%%%%%%%%%%%%%%%%%%%%%%
\subsection{The distribution of $(\T_{k,k'},\mu_{k,k'}$)}
\label{sec:loinonconservatif}
%%%%%%%%%%%%%%%%%%%%%%%%%%%%%%%%%%%%%
In Section \ref{sec:cvd}, the distribution of $\T_k$ was obtained by using the main theorem of \cite{HM12}. We would like to do the same with $\T_{k,k'}$, but the issue is that the results of \cite{HM12} are restricted to conservative fragmentations. The aim of this section is therefore to concisely show that the arguments used in \cite{HM12} still apply in our context and prove the last part of Theorem \ref{thjoint}: that $(\T_{k,k'},d,\rho,\mu_{k,k'})$ has the distribution of a fragmentation tree with index $-1/k$ and dislocation measure $\nu^{\downarrow}_{k,k'}$ (defined in Theorem \ref{thjoint}). For reference, we let $(\T^0,d^0,\rho^0,\mu^0)$ be such a fragmentation tree.

To prove this identity in distribution, we will look at the finite-dimentional marginals of $(\T_{k,k'},\mu_{k,k'})$ and $(\T^0,\mu^0)$, in the traditional sense of finite-dimentional marginals for measured $\R$-trees. Specifically, for all integers $l$, let $X_1,\ldots,X_l$ be i.i.d. points of $\T_{k,k'}$ distributed according to $\mu_{k,k'}$ conditionally on $(\T_{k,k'},\mu_{k,k'})$ and $X_1^0,\ldots,X_l^0$ be i.i.d. points of $\T^0$ distributed according to $\mu^0$ conditionally on $(\T^0,\mu^0)$. We will prove that the finite metric spaces $\big((\rho,X_1,\ldots,X_l),d\big)$ and $\big((\rho^0,X_1^0,\ldots,X^0_l),d^0\big)$ have the same distribution for all $l\in\N$, and this will imply that $(\T_{k,k'},\mu_{k,k'})$ and $(\T^0,\mu^0)$ also have the same distribution, because each tree is the completion of the union of its finite-dimensional marginals (this is true because both measures $\mu_{k,k'}$ and $\mu^0$ are fully supported on their respective trees, which itself is true because $(\T_k,\mu_k)$ and $(\T^0,\mu^0)$ are fragmentation trees with infinite dislocation measures -- see \cite[Theorem 1]{HM04}, and moreover the property of having full support is conserved under projection). We already know that $n^{-1/k}\big(T_n(k,k'),\mu_n(k,k')\big)$ converges to $(\T_{k,k'},\mu_{k,k'})$ in probability for the GHP-topology, which implies the convergence of the finite-dimensional marginals in distribution. It will therefore suffice to show that the finite-dimensional marginals of $n^{-1/k}\big(T_n(k,k'),\mu_n(k,k')\big)$ converge to those of $(\T^0,\mu^0)$. This is essentially Proposition 30 of \cite{HM12}: we will show by induction on $l\geq 1$ that, if, for all $n\in\Z_+$, $(X_1(n),\ldots,X_l(n))$ are, conditionally on $\big(T_n(k,k'),\mu_n(k,k')\big)$, independent points of $T_n(k,k')$ with distribution $\mu_n(k,k')$, then the space $\big((\rho,X_1(n),\ldots,X_l(n)),d\big)$ converges in distribution to $\big((\rho^0,X_1^0,\ldots,X^0_l),d^0\big)$.

We start with the case where $l=1$, where, just as in Lemma 28 in \cite{HM12}, the result is a consequence of Theorem 2 of \cite{HM11}. For $n\in\Z_+$, let $X(n)\in T_n(k,k')$ have distribution $\mu_n(k,k')$. Note that the only information contained in the metric space $\big((\rho,X(n)),d\big)$ is the the height of $X(n)$, so we set out to prove the convergence in distribution of this height, when rescaled, to the height of a point of $\T^0$ with distribution $\mu^0$. This height can be explicited with the help of a non-increasing Markov chain: follow the path from the root to $X(n)$ and, at every point, out of the $k'$ subtrees of $T_n(k,k')$ rooted at that point, take the $\mu_n(k,k')$-mass of the one containing $X(n)$, multiplied by $(k-1)n+1$ to make an integer, with two exceptions: for the root, where there is only one subtree, take the value $(k-1)n+1$, and when we reach $X(n)$, take the value $0$. This is indeed a decreasing Markov chain on $\Z_+$ because of the Markov branching property, its initial value is $(k-1)n+1$ and its transition probabilities $p_{a,b}$ (with $b\leq a$) do not depend on $n$ and are best described by the following. First let $(q'_{n-1})^{\downarrow}$, in analogous fashion to $\bar{q}'_{n-1}$ from Section \ref{sec:asympqn}, be the distribution of the reordering of the first $k'$ terms of a sequence with distribution $q_{n-1}$. When the Markov chain is at $(k-1)n+1\in \mathbb N$, take a variable $\lambda$ with distribution $(q'_{n-1})^{\downarrow}$, jump to value $(k-1)\lambda_i+1$ with probability $((k-1)\lambda_i+1)/((k-1)n+1)$ and jump to $0$ with the complementary probability. The height of $X(n)$ is then the time at which this Markov chain reaches $0$. For any function $f$, setting $r_n=(k-1)n+1$, we obtain
\begin{align*}
\sum_{b=0}^{r_n} &f\bigg(\frac{b}{r_n}\bigg) p_{r_n,b}=\\
                 &\sum_{\lambda=(\lambda_1,...,\lambda_{k'}) \in \mathbb Z_+^{k'}:\sum_{i=1}^{k'}\lambda_i \leq n-1}^{k'}(q'_{n-1})^{\downarrow}(\lambda)\Bigg(\sum_{i=1}^{k'} \Big(\frac{r_{\lambda_i}}{r_n}\Big)f\Big(\frac{r_{\lambda_i}}{r_n}\Big)+\Bigg(1-\sum_{i=1}^{k'}\frac{r_{\lambda_i}}{r_n}\Bigg)f(0)\Bigg).
\end{align*}
With this and Corollary \ref{coro:cvmes}, it then follows that the measure $n^{1/k}(1-x)\sum_{b=0}^{r_n} p_{r_n,b} \delta_{b/{r_n}}(\mathrm{d}x)$ converges weakly to $$\int_{\s_{k'}} \bigg(\sum_{i=1}^{k'}(1-s_i)s_i\delta_{s_i} +\bigg(1-\sum_{i=1}^{k'}s_i \bigg)\delta_0\bigg) \nu_{k,k'}^{\downarrow}(\mathrm{d}\mathbf{s}).$$ Theorem 2 of \cite{HM11} is then applicable and shows that, when renormalized by $n^{-1/k}$, the height of $X(n)$ does converge in distribution to the height of a point of $\T^0$ with distribution $\mu^0$, which can be written $\int_0^\infty \e^{-\xi_t/k}\mathrm{d}t$, where $(\xi_s)_{s\geq0}$ is a subordinator with Laplace exponent defined for $q\geq 0$ by $\int_{\s_{k', \leq}} (1-\sum_{i=1}^{k'}s_i^{q+1})  \nu^{\downarrow}_{k,k'}(\mathrm d \mathbf{s})$.

Now take $l>1$ and assume that the convergence of $l'$-dimensional marginals has been proven for all $l'< l$. For $n\in\N$, take $(Y_1(n),\ldots,Y_l(n))$ to be i.i.d. uniform leaves of $T_n(k)$, conditioned to being all different, an event which has probability tending to $1$, and let $(X_1(n),\ldots,X_l(n))$ be their projections on $T_n(k,k')$. Lemma 29 of \cite{HM12} still applies, namely: the height of the branch point of $(X_1(n),\ldots,X_l(n))$ converges in distribution to the height of the branch point of $l$ variables in $\T^0$ which are i.i.d. with distribution $\mu^0$  
and this convergence holds jointly with the masses of the subtrees containing the $(X_i(n))_{i\in[l]}$ above this branch point, as well as the allocations of the $(X_i(n))_{i\in[l]}$ in these subtrees. Applying the induction hypothesis and the self-similarity property of $\T^0$ on each of these subtrees then ends the proof.
The fact that Lemma 29  of \cite{HM12} still holds will be left for the reader to check: it hinges on the sublemmas 26 and 27 of \cite{HM12}, which require some modification to account for non-conservativeness. These details will appear in \cite{these}. $\hfill \square$
%%%%%%%%%%%%%%%%%%%%%%%%%%%%%%%%%%%%%
\subsection{Proof of (\ref{lawsubtree})}
\label{sec:ML}
%%%%%%%%%%%%%%%%%%%%%%%%%%%%%%%%%%%%%

For $n\geq 0$, let $I_n$ denote the number of internal nodes of $T_n(k)$ which are in $T_{n}(k,k')$. 

\begin{lemma} 
\label{Lemmalaw}
One has
$$\left(T_n(k,k'), n\geq 0 \right)=\big(\tilde T_{I_n}(k'), n \geq 0 \big),$$ where $(\tilde T_{i}(k'),i\geq 0)$  is a sequence distributed as $(T_{i}(k'),i\geq 0)$  and independent of $(I_n,n\geq 0)$. Moreover, $(I_n,n\geq 0)$ is a Markov chain with transition probabilities
$$
\mathbb P\left(I_{n+1}=i+1 \ | \ I_n=i \right)=1-\mathbb P\left(I_{n+1}=i \ | \ I_n=i \right)=\frac{k' i+1}{kn+1},
$$
and as a consequence,
$$
\frac{I_n}{n^{k'/k}} \overset{\mathrm{a.s.}} \longrightarrow M_{k'/k,1/k},
$$
where the limit is a $(k'/k,1/k)$-generalized Mittag-Leffler random variable. 
\end{lemma}

\bigskip

We recall that a generalized Mittag-Leffler random variable $M_{\alpha,\theta}$ with parameters $\alpha \in (0,1)$ and $\theta>-\alpha$ has its distribution characterized by its positive moments, given by
\begin{equation*}
\mathbb E\left[M^p_{\alpha,\theta}\right]=\frac{\Gamma(\theta+1)\Gamma(\theta/\alpha+p+1)}{\Gamma(\theta/\alpha+1)\Gamma(\theta+p\alpha+1)}, \quad p \geq 0.
\end{equation*}

\bigskip

\noindent \textbf{Proof.} This proof is very similar to  those of Lemma 8 and Lemma 9 of \cite{CH13}. Given $T_{i}(k)$ and $T_{i}(k, k')$ for $0 \leq i \leq n$, the new node added to get $T_{n+1}(k)$ from $T_n(k)$ will belong to $T_{n+1}(k, k')$ if  and only if the selected edge is in $T_{n}(k, k')$, which occurs with probability $(k' I_n+1)/(kn+1)$ since  $k' I_n+1$ is the number of edges of $T_{n}(k, k')$ and $kn+1$ that of $T_{n}(k)$. Moreover, conditionally to the fact that this new node belongs to $T_{n+1}(k, k')$, it is located uniformly at random on one of the edges of  $T_n(k,k')$, independently of the whole process $(I_n, n \geq 0)$ and of  $T_i(k,k')$ for $0 \leq i \leq I_{I_n}^{-1}-1$ where $I_m^{-1}:=\inf\{n \geq 0:I_n=m\}$, $m\geq 0$.
From this, it should be clear that the process defined for all $i \geq 0$ by $$\tilde T_i(k')=T_{I_i^{-1}}(k, k')$$ 
is distributed as $(T_i(k'),i \geq 0)$ and independent of $(I_n, n \geq 0)$. Moreover, we have that $T_n(k,k')=\tilde T_i(k')$ if  $I_n=i$, hence $T_n(k,k')=\tilde T_{I_n}(k')$. 

Lastly, the few lines above show that $(I_n,n \geq 0)$ is a Markov chain with the expected transition probabilities. It turns out that these probabilities are identical to those of the number of tables in a $(k'/k,1/k)$ Chinese restaurant process. Therefore, using again Theorem 3.8 in \cite{PitmanStFl}, $n^{-k'/k}I_n$ converges almost surely towards a $(k'/k,1/k)$-generalized Mittag-Leffler random variable. 
$\hfill \square$

\bigskip

\noindent \textbf{Proof of (\ref{lawsubtree}).} This is a straightforward consequence of the joint convergence in probability settled in Theorem \ref{thjoint} and of Lemma \ref{Lemmalaw}. Indeed, we know that
$$
\left(\frac{T_{n}(k)}{n^{1/k}}, \frac{T_{n}(k, k')}{n^{1/k}}\right) \overset{\mathbb P}{\underset{n \rightarrow \infty} \rightarrow} \left(\mathcal T_k , \mathcal T_{k,k'} \right)
$$
(we are not interested in measures on trees here). Then, for $n \geq 1$,
$$
\frac{T_{n}(k, k')}{n^{1/k}}=\frac{T_{n}(k, k')}{I_{n}^{1/k'}} \times \left(\frac{I_{n}}{n^{k'/k}}\right)^{1/k'}.
$$
On the one hand, the left hand side converges in probability towards $\mathcal T_{k,k'}$. On the other hand, by Lemma \ref{Lemmalaw} and since $I_{n}$ converges a.s. to $+\infty$, 
$$ 
\frac{T_{n}(k, k')}{I_{n}^{1/k'}}  \overset{\mathbb P}{\underset{n \rightarrow \infty} \rightarrow} \tilde{\mathcal T}_{k'},
$$ 
where $ \tilde{\mathcal T}_{k'}$ is distributed as $\mathcal T_{k'}$.
Moreover this holds independently of the a.s. convergence of $I_{n}/n^{k'/k}$ towards the generalized Mittag-Leffler r.v. $M_{k'/k,1/k}$. The result follows by identification of the limits.
$\hfill \square$

\bigskip

Actually, following the ideas of the proof of Theorem  15 in \cite{CH13}, we can reinforce the identity in distribution (\ref{lawsubtree}) in an identity of distribution of measured trees.
For the sake of brevity, we do not state this additional result here and refer the interested reader to \cite[Theorem 15]{CH13} for a similar result in the context of stable Lévy trees, which can easily be adapted to our context.

%%%%%%%%%%%%%%%%%%%%%%%%%%%%%%%%%%%%%%%%%%%%%
\subsection{Extracting a tree with distribution $\mathcal T_{k'}$ from $\mathcal T_{k}$}
\label{sec:extraction}
%%%%%%%%%%%%%%%%%%%%%%%%%%%%%%%%%%%%%%%%%%%%%

We know from the discrete approximation that there is a subtree of $\mathcal T_{k}$ which is distributed as \linebreak $M^{1/k'}_{k'/k,1/k}\cdot \mathcal T_{k'}$ (or, equivalently, as a fragmentation tree with index $-1/k$ and dislocation measure $\nu^{\downarrow}_{k,k'}$). Our goal is now to explain how to extract such a tree directly from $\mathcal T_{k}$. Our approach strongly relies on the fact that $(\mathcal T_k,\mu_k)$ is a fragmentation tree. 

As a fragmentation tree, $\mathcal T_{k}$ has  a countable number of branch points, almost surely. We denote this set of branch points $\{b(n),n\in \mathbb N\}$.
For each $n\in \mathbb N$, we recall that $$\mathcal T_{b(n)}=\{v \in \mathcal T_k: b(n) \in [[\rho,v]]\}$$ is the subtree of descendants of $b(n)$ ($\rho$ denotes the root of $\mathcal T_k$). Since $\mathcal T_k$ is $k$-ary, the set $\mathcal T_{b(n)} \backslash\{b(n)\}$ has exactly $k$ connected components. We label them as follows: $\mathcal T_{b(n),1}$ is the connected component with the largest $\mu_k$-mass, $\mathcal T_{b(n),2}$ is the connected component with the second largest $\mu_k$-mass, and so on (if two or more trees have the same mass, we label them randomly). 

For $n \in \mathbb N$ and $i=1,...,k$, let $$s_i(n)=\frac{\mu_k(\mathcal T_{b(n),i})}{\mu_k(\mathcal T_{b(n)})}.$$ Almost surely, for all $n \in \mathbb N$, these quotients are well-defined, strictly positive and sum to 1 (see \cite{HM04, Steph13}). We then mark the sequences $\mathbf s(n)$, \textit{independently} for all $n \in \mathbb N$, by associating  to each sequence $\mathbf{s} \in \mathcal S_k$ an element $\mathbf{s}^* \in \mathcal S_{k', \leq}$ by deciding that for all $1 \leq i_1<...<i_{k'}\leq k$
\begin{equation}
\label{mark}
(s^{*}_1,...,s^{*}_{k'})=(s_{i_1},...,s_{i_{k'}}) \text{ with probability } 
\frac{(k'-1)!(k-k')!}{(k-1)!} \frac{\sum_{j \in \{i_1,...,i_k'\}} \prod_{1\leq i \neq j \leq k}(1-s_i)}{\sum_{j=1}^k \prod_{1\leq i \neq j \leq k}(1-s_i)}.
\end{equation}
This means that we attribute a weight $ \prod_{i \neq j}(1-s_i)$ to the $j$th term of the sequence $\mathbf s$, for all $1\leq j \leq k$, and then choose at random a  $k'$-uplet of terms (with strictly increasing indices) with a probability proportional to the sum of their weights. One can easily check that, for any sequence $\mathbf s$, the quotient in (\ref{mark}) indeed defines a probability distribution since 
$(k-1)!/((k'-1)!(k-k')!)$ is the number of  $k'$-uplets $(i_1,...,i_{k'})$, with $1 \leq i_1<...<i_{k'}\leq k$, containing a given integer $j \in \{1,...,k\}$.
For $n \in \mathbb N$, if $(s^{*}_1(n),...,s^{*}_{k'}(n))=(s_{i_1}(n),...,s_{i_{k'}}(n))$, we then let $${\mathcal T}^{*}_{b(n)}=\bigcup_{j \in \{1,...,k\} \backslash \{i_1,...,i_{k'}\}}\mathcal T_{b(n),j}.$$ Finally we set 
\begin{equation}
\label{def:Tkk'}
\mathcal T_{k,k'}^*=\mathcal T_k \backslash \bigcup_{n \in \mathbb N} {\mathcal T}^{*}_{b(n)}. 
\end{equation}
In words, $\mathcal T_{k,k'}^*$ is obtained from $\mathcal T_k$ by removing all groups of trees ${\mathcal T}^{*}_{b(n)}$ for $n \in \mathbb N$. This tree (which is well-defined almost surely) has the required distribution: 

\begin{prop} 
\label{prop:prun}
The tree $\mathcal T_{k,k'}^*$ is a non-conservative fragmentation tree, with index of self-similarity $-1/k$ and dislocation measure $\nu^{\downarrow}_{k,k'}$.
\end{prop}

\begin{proof}
Let $\mathcal P_{\mathbb N}$ denote the set of partitions of $\mathbb N$, and equip it with the distance $$d_{\mathcal P_{\mathbb N}}(\pi,\pi')=\exp\big(-\sup\{k \geq 1:\pi |_{[k]}=\pi' |_{[k]}\}\big)$$ where $(\pi,\pi')$ denote any pair of partitions of $\mathbb N$, and $\pi |_{[k]}$, $\pi' |_{[k]}$ their respective restrictions to the $k$ first positive integers.
Let  then $(A_i)_{i\in\N}$ be an exchangeable sequence of leaves of $\mathcal T_k$ directed by $\mu_k$. 
Define from it a càdlàg partition-valued process $(\Pi(t))_{t\geq0}$ by declaring, for $t\geq0$, that two different integers $i$ and $j$ are in the same block of $\Pi(t)$ if $A_i$ and $A_j$ are in the same connected component of $\{x\in\T_k,ht(x)>t\}$. According to \cite[Section 2.3]{HM04} or \cite[Proposition 3.1]{Steph13}, this is a partition-valued fragmentation process with dislocation measure $\nu^{\downarrow}_k$ and self-similarity index $-1/k$ (and no erosion). We thus know thanks to \cite{BertoinHomogeneous, BertoinSSF} that, up to a family of suitable time-changes (that we do not recall here), the process $\Pi$ can be constructed from a Poisson point process $\big((\Delta(s),i(s)),{s\geq0}\big)$ on $\mathcal P_{\mathbb N} \times \mathbb N$, with  intensity measure $\kappa_{\nu^{\downarrow}_k} \otimes \#$, where $\#$ denotes the counting measure on $\mathbb N$ and $\kappa_{\nu^{\downarrow}_k}$ is a $\sigma$-finite measure on $\mathcal P_{\mathbb N}$ defined by 
$$
\kappa_{\nu^{\downarrow}_k}(\mathrm d \pi)=\int_{\mathcal S_k} \kappa_{\mathbf s}(\mathrm d \pi) \nu^{\downarrow}_{k}(\mathrm d \mathbf s)
$$
where $\kappa_{\mathbf s}$ denotes the exchangeable probability on $\mathcal P_{\mathbb N}$ with paintbox $\mathbf s$.

The connection between the Poisson point process $\big((\Delta(s),i(s)),{s\geq0}\big)$ and the tree $\mathcal T_k$ can be partially  summarized as follows (the following assertions hold almost surely). There is a bijection between the set  of atoms of this Poisson point process  and the set of branch points of $\mathcal T_k$. For each atom $(\Delta(s),i(s))$, let $b(n_{\Delta(s),i(s)})$ be the corresponding branch point, with $n_{\Delta(s),i(s)} \in \mathbb N$. There exists then an infinite subsequence  $(A_{i_m},m \in \mathbb N)$ of $(A_i,i\in \mathbb N)$ composed by the leaves that belong to $\mathcal T_{b(n_{\Delta(s),i(s)})}$. Then, two integers $m_1$ and $m_2$ are in a same block of $\Delta(s)$ if and only if $A_{i_{m_1}}$ and $A_{i_{m_2}}$ are in a same subtree $\mathcal T_{b(n_{\Delta(s),i(s)}),j}$ of $\mathcal T_{b(n_{\Delta(s),i(s)})}$ for some $1\leq j \leq k$. For more details (the roles of the integers $i(s)$ and time $s \geq 0$) we refer to \cite{HM04}.
We decide to label the $k$ blocks of $\Delta(s)$ according to the indices of the corresponding subtrees  $\mathcal T_{b(n_{\Delta(s),i(s)}),j}, 1 \leq j\leq k$. 

We then mark the Poisson point process as follows: for each atom $(\Delta(s),i(s))$, we extract randomly  $k'$ blocks of $\Delta(s)$ by setting $$(\Delta^{*}_1(s),...,\Delta^{*}_{k'}(s))=(\Delta_{i_1}(s),...,\Delta_{i_{k'}}(s))$$ if $$(s^{*}_1(n_{\Delta(s),i(s)}),...,s^{*}_{k'}(n_{\Delta(s),i(s)}))=(s_{i_1}(n_{\Delta(s),i(s)}),...,s_{i_{k'}}(n_{\Delta(s),i(s)})),$$
where the sequence $\mathbf s^*(n_{\Delta(s),i(s)})$ is  the one obtained from $\mathbf s(n_{\Delta(s),i(s)})$ by the marking procedure (\ref{mark}). Then, we make $\Delta^*(s)$ into a partition of $\N$ with dust by putting every integer which is not originally in a block $\Delta_1^*(s),...,\Delta_{k'}^*(s)$ into a singleton. The process $\big((\Delta^*(s),i(s)),{s\geq0}\big)$ is  therefore a marked Poisson point process with intensity 
$\kappa_{\nu^{\downarrow,*}_{k}} \otimes \#$, where
$$
\kappa_{\nu^{\downarrow,*}_{k}}(\mathrm d \pi)=\int_{\mathcal S_{k',\leq}} \kappa_{\mathbf s}(\mathrm d \pi) {\nu^{\downarrow,*}_k}(\mathrm d \mathbf s) \quad \text{and} \quad  \int_{\mathcal S_{k',\leq}} f (\mathbf s) \nu^{\downarrow,*}_{k}(\mathrm d \mathbf s)=\int_{\mathcal S_k} \mathbb E[f (\mathbf s^*)] {\nu^{\downarrow}_k}(\mathrm d \mathbf s),
$$
for all suitable test functions $f$. Now, the key-point is that 
$$
 \nu^{\downarrow,*}_{k}= \nu^{\downarrow}_{k,k'}.
$$
This is easy to check by using the definitions of $\nu^{\downarrow}_{k}$, $\nu^{\downarrow}_{k,k'}$ and of the marking procedure (\ref{mark}), together with the identity (\ref{beta}). The details of this calculation are left to the reader.

To finish, let $\Pi^*$ be the $(-1/k, \nu^{\downarrow}_{k,k'})$-fragmentation process derived from the Poisson point process $\big((\Delta^*(s),i(s)),{s\geq0}\big)$. For all $i\in\N$, let $D^*_i=\inf\{t\geq0, \{i\}\in\Pi^*(t)\}$ and note that $D^*_i \leq D_i$, where $D_i:=\inf\{t\geq0, \{i\}\in\Pi(t)\}$ is the height of $A_i$ in $\mathcal T_k$.
Let then $A_i^*$ be the unique point of $\T_k$ belonging to the geodesic $[[\rho,A_i]]$ which has height $D^*_i$. It is not hard to see that   $\T^*_{k,k'}$, defined by (\ref{def:Tkk'}), is the closure of the subtree $\cup_{i \geq 1} [[\rho,A_i^*]]$ of $\mathcal T_k$ spanned by the root and all the vertices $A^*_i$ (almost surely). But by definition (see \cite{Steph13}), this closure is the genealogy tree of $\Pi^*$. Thus $\T^*_{k,k'}$ has the distribution of a $(-1/k, \nu^{\downarrow}_{k,k'})$-fragmentation tree.
\end{proof}

\appendix

%%%%%%%%%%%%%%%%%%%%%%%%%%%%%%%%%%%%%%%%%%%%%%%%%%%%%%%%%%
\section{GHP-convergence of discrete trees with edge-lengths}
\label{sec:A1}
%%%%%%%%%%%%%%%%%%%%%%%%%%%%%%%%%%%%%%%%%%%%%%%%%%%%%%%%%%

Let $T$ be a rooted finite graph-theoretical tree: we think of it as a set of vertices equipped with a set of edges $E$. For any strictly positive function $l$ on $E$, we let $T_l$ be the $\R$-tree obtained from $T$ by considering every edge $e$ as a line segment with length $l(e)$, and call $d_l$ its metric. In the following, we will sometimes need to embed such trees in $\ell^1$, simultaneously for several different functions $l$. In order to do this in a way which lets us simply compare the trees, we first label the leaves of $T$ (labels that are of course transposed to $T_l$ for all $l$) and then use the stick-breaking construction of Aldous, as recalled at the end of the proof of Proposition \ref{prop:cvmarginalesGH}. 

\begin{lemma}
\label{lem:cvGHarbresdiscrets}  
Let $(l_n)_{n\in\N}$ be a sequence of strictly positive functions on $E$ and assume that, for all $e\in E$, $l_n(e)$ which converges to a strictly positive number $l(e)$ as $n$ goes to infinity. We then have
	\[T_{l_n} \underset{n\to\infty}{\overset{\mathrm{GH}}\longrightarrow} T_l.
\]
With the extra assumption that no vertices of $T$ have degree $2$, it is then in fact sufficient to know that, for all leaves $L$ and $L'$ of $T$, $d_{l_n}(L,L')$ and $d_{l_n}(\rho,L)$ converge respectively to $d_l(L,L')$ and $d_l(\rho,L)$.

Moreover, with the $\ell^1$-embedded versions of the trees, we have Hausdorff convergence in $\ell^1$.

\end{lemma}

\begin{proof} For the first point, we just need to prove the  Hausdorff convergence in $\ell^1$ of the embedded versions of the trees. For this, one only needs to notice that $$d_{\ell^1,\mathrm{H}}(T_{l_n},T_l)\leq \sum_{e\in E} |l_n(e)-l(e)|,$$ which converges to $0$. The proof of the second point is merely a matter of noticing that, if we know the distances between the leaves (including the root), we can recover the whole metric on a tree.
\end{proof}

\medskip

We now recall a result of \cite{Steph13} which gives us a practical way of building measures on a compact $\R$-tree. Let $\T$ be any compact rooted tree and $m$ a nonnegative function on $\T$. We say that $m$ is decreasing if, for all $x$ and $y$ in $\T$ with $x\in[[\rho,y]]$, we have $m(x)\geq m(y)$. In this case, one can define a left-limit $m(x^-)$ of $m$ at $x$ as 
	\[m(x^-)=\underset{\underset{z\in [[\rho,x[[}{z\to x}}\lim m(z)
\]
(in the case of the root, we simply let $m(\rho^-)=m(\rho)$). One can also define what we call the \emph{additive right-limit}. Recall that $\T_x$ is the subtree of descendants of $x$. Suppose first that $x$ is not a leaf. By compactness,  the space $\T_x\setminus\{x\}$ has countably many connected components, say $(\T_i)_{i\in S}$ for a finite or countable set $S$. Let, for all $i\in S$, $x_i\in\T_i$. We then set 
	\[m(x^+)= \underset{i\in S}\sum \,\underset{\underset{z\in ]]x,x_i]]}{z\to x}}\lim m(z).
\]
If $x$ is a leaf, then we let $m(x^+)=0$.

\begin{lemma}[{\cite[Proposition 2.7]{Steph13}}]
\label{lem:mesuresarbres}
Assume that, for all $x\in\T$, we have $m(x^-)=m(x)\geq m(x^+)$. Then there exists a unique measure $\mu$ on $\T$ such that, for all $x$ in $\T$, we have
	\[\mu(\T_x)=m(x).
\]
We then also have
	\[\mu(\{x\})=m(x)-m(x^+), \quad \forall x \in \T.
	\]
\end{lemma}

\smallskip

Note that the converse is also true (but elementary): for any finite measure $\mu$ on $\mathcal T$, the function $m$ defined by $m(x)=\mu(\T_x)$ satisfies $m(x^-)=m(x)\geq m(x^+)$ for all $x\in \T$.

\medskip
Our next result shows that this theory is compatible with the convergence of discrete trees. Return to the assumptions of Lemma \ref{lem:cvGHarbresdiscrets}: $T$ is a finite graph-theoretical tree and, for all $n$, we have a length functions $l_n$ on the set of edges. The sequence $(l_n(e))_{n\in\N}$ is assumed to converge to a strictly positive $l(e)$ for every edge $e$ and then, every tree being embedded in $\ell^1$ with the stick-breaking method, $T_{l_n}$ converges, in the Hausdorff sense for compact subsets of $\ell^1$, to $T_l$.

\begin{lemma}
\label{lem:cvGHParbresdiscrets} 
For $n\in\N$, let $\mu_n$ be a probability measure on $T_{l_n}$ with $m_n$ the corresponding decreasing function. Let $S$ be any dense subset of $T_l$, and assume that, for all $x\in S$, there exists a sequence $(x_n)_{n\in\N}$, such that
\begin{itemize}
\item[$\bullet$] $x_n\in T_{l_n}$ for all $n$, $x_n$ converges to $x$ as $n$ goes to infinity,
\item[$\bullet$] $(T_{l_n})_{x_n}$ converges to $(T_l)_x$ in the Hausdorff sense,
\item[$\bullet$] $m_n(x_n)$ converges to a number we call $f(x)$.
\end{itemize}
We then have
	\[(T_{l_n},\mu_n) \underset{n\to\infty}{\overset{\mathrm{GHP}}\longrightarrow} (T_l,\mu),
\]
where $\mu$ is the unique probability measure on $T_l$ such that, for all $x\in T_l$, $\mu((T_l)_x)=f(x^-)$, and $f(x^-)$ is defined as
\begin{equation}
\label{limgaucheS}
f(x^-)=\underset{\underset{y \in S\cap[[\rho,x[[} {y\to x}} \lim f(y),
\end{equation}
and $f(\rho^-)=1$.
More precisely, since we consider the versions of the trees embedded in $\ell^1$, we have Hausdorff convergence of the sets and Prokhorov convergence of the measures.
\end{lemma}

\begin{proof} Since $T_l$ is compact, $(\cup T_{l_n}) \cup T_l$ also is and Prokhorov's theorem ensures us that a subsequence of $(\mu_n)_{n\in\N}$ converges weakly. Without loss of generality, we can assume therefore that $(\mu_n)_{n\in\N}$ converges to a measure $\mu$ on $T_l$. We will show that $\mu$ must be as explicited in the statement of the lemma. This will be done by showing the following double inequality for all $x\in S$, which is inspired by the Portmanteau theorem,
\begin{equation}\label{double}
\mu\big((T_l)_x\setminus\{x\}\big)\leq f(x)\leq \mu\big((T_l)_x\big).
\end{equation}
We start by showing the right part of (\ref{double}): $f(x)\leq\mu((T_l)_x)$. Let $\varepsilon>0$, by Hausdorff convergence in $\ell^1$, for $n$ large enough, we have $(T_{l_n})_{x_n}\subset ((T_l)_x)^{\varepsilon}$, where $A^{\varepsilon}$ is the closed $\varepsilon$-enlargement of a set $A$. Since we also have $d_{\ell^1,\mathrm P}(\mu_{n},\mu)\leq \varepsilon$ for $n$ large enough, we obtain
	\[\mu_n\big((T_{l_n})_{x_n}\big)\leq \mu_n\Big(\big((T_l)_x\big)^{\varepsilon}\Big)\leq \mu\big(((T_l)_x)^{2\varepsilon}\big) + \varepsilon,
\] and making $n$ tend to infinity then gives us
	\[f(x)\leq \mu\Big(\big((T_l)_x\big)^{2\varepsilon}\Big) + \varepsilon.
\] Letting $\varepsilon$ tend to $0$ and using the fact that $(T_l)_x$ is closed gives us $f(x)\leq \mu\big((T_l)_x\big)$.

A similar, slightly more involved argument will show that $\mu\big((T_l)_x\setminus\{x\}\big) \leq f(x)$ for $x\in S$. Let $x\in S$ and let $d+1$ be its degree (there is nothing to say if $x$ is a leaf or the root). Let $T^1,\ldots,T^d$ be the connected components of $(T_l)_x\setminus\{x\}$ and let $y^1,\ldots,y^d$ be any points of $T^1,\ldots,T^d$ which also are in $S$. We give ourselves the corresponding sequences $(y^1_n)_{n\in\N},\ldots,(y^d_n)_{n\in\N}$. Take $\varepsilon>0$, we have, for $n$ large enough,
	\[\cup_{i=1}^d \big((T_l)_{y^i}\big) \subset \cup_{i=1}^d \big((T_{l_n})_{y^i_n}\big)^{\varepsilon},
\]
and therefore, using the Prokhorov convergence of measures, for possibly larger $n$,
\begin{align*}
\mu\Big(\cup_{i=1}^d \big((T_l)_{y^i}\big)\Big)&\leq \mu\Big(\cup_{i=1}^d \big((T_{l_n})_{y^i_n}\big)^{\varepsilon}\Big) \\
                    &\leq \mu_n \Big(\cup_{i=1}^d \big((\T_{l_n})_{y^i_n}\big)^{2\varepsilon}\Big)+\varepsilon.
\end{align*}
Since $\mu_n$ is supported on $\mathcal T_n$, if we take $2\varepsilon < \max_{1 \leq i \leq d} d(y^i,x)$, and $n$ large enough, we obtain
	\[\mu_n \Big(\cup_{i=1}^d \big((T_{l_n})_{y^i_n}\big)^{2\varepsilon}\Big)\leq \mu_n \big((T_{l_n})_{x_n}\big),
\]
thus giving us
	\[\mu\Big(\cup_{i=1}^d \big((T_l)_{y^i}\big)\Big) \leq \mu_n \big((\T_{l_n})_{x_n}\big) +\varepsilon.
\]
Letting $n$ tend to infinity and then $\varepsilon$ tend to $0$, we obtain
	\[\mu\Big(\cup_{i=1}^d \big((T_l)_{y^i}\big)\Big)\leq f(x),
\]
and finally we let all the $y^i$ tend to $x$, which makes the left-hand side tend to $\mu\big((T_l)_x\setminus\{x\}\big)$.

Having proved (\ref{double}), we only need to check that, calling $m$ the decreasing function associated to $\mu$, $m$ is equal to the left-limit of $f$ as defined in (\ref{limgaucheS}), which is immediate: let $x\in T_l\backslash\{\rho\}$ and evaluate (\ref{double}) at a point $y\in [[\rho,x[[\cap S$. By left-continuity of $m$, if we let $y$ tend to $x$, both the left and right members converge to $m(x)$, while the middle one converges to $f(x^-)$, which ends the proof.
\end{proof}

%%%%%%%%%%%%%%%%%%%%%%%%%%%%%%%%%%%%%%%%%%%%%%%%%%%%%%%%%%
\section{Trees, subtrees and projections}
\label{sec:A2}
%%%%%%%%%%%%%%%%%%%%%%%%%%%%%%%%%%%%%%%%%%%%%%%%%%%%%%%%%%
Let $(\T,d,\rho)$ be a compact and rooted $\R$-tree and $\T'$ be a compact and connected subset of $\T$ containing $\rho$. 
The boundary $\partial\T'$ of $\T'$ in $\T$ is then finite or countable. 
We recall that  $\T_x$ denotes the subtree of $\T$ rooted at $x$, $\forall x \in \mathcal T$, and similarly let $\mathcal T'_x$ denote the subtree of $\T'$ rooted at $x$, for $\ x \in \mathcal T'$. We then have
	\[\T=\T'\cup\bigcup_{x\in\partial\T'} \T_x
\]
with only the elements of $\partial\T'$ being counted multiple times in this union.

For $x\in\T$, there exists a highest ancestor of $x$ which is in $\T'$. We call it $\pi(x)$. The map $\pi$ is called the projection from $\T$ on $\T'$. For technical reason, we consider it as a map from $\T$ to $\T$, so that, for any measure $\mu$ on $\T$, $\pi_* \mu$ defines a measure on $\T$ (that only charges $\T'$).

\begin{lemma}
\label{lem:projmesure}
For any probability measure $\mu$ on $\T$, $\pi_*\mu$ is the unique probability measure $\nu$ on $\T'$ which satisfies
	\[\forall x\in \T', \quad \nu(\T'_x)=\mu(\T_x).
\]
\end{lemma}
\begin{proof} The fact that $\pi_*\mu$ satisfies the relation comes from the fact that, for all $x\in\T'$, we have $\T_x=\pi^{-1}(\T'_x)$. Uniqueness is a consequence of Lemma \ref{lem:mesuresarbres}.
\end{proof}

\begin{lemma}
\label{lem:projlip}
The map $\pi$ is $1$-Lipschitz whether one considers points of $\T$, the Hausdorff distance between compact subsets of $\T$ or the Prokhorov distance between probability measures on $\T$:
\begin{itemize}
\item $\forall x,y\in \T, d(\pi(x),\pi(y))\leq d(x,y),$ \\
\item for $A$ and $B$ non-empty compact subsets of $\T$, $d_{\T,\mathrm{H}}(\pi(A),\pi(B))\leq d_{\T,\mathrm{H}}(A,B),$ \\
\item  for any two probability measures $\mu$ and $\nu$ on $\T$, $d_{\T,\mathrm{P}}(\pi_*\mu,\pi_*\nu)\leq d_{\T,\mathrm{P}}(\mu,\nu)$.
\end{itemize}
\end{lemma}

\begin{proof} Let $x$ and $y$ be elements of $\T$. Assume first that $x\in [[\rho,y]]$. If both of them are in $\T'$ then $\pi(x)=x$ and $\pi(y)=y$, while if they are both not in $\T'$, then $\pi(x)=\pi(y)$. If $x$ is in $\T'$ but $y$ is not, then $\pi(y)\in [[x,y]]$. In all these three cases, we have $d(\pi(x),\pi(y))\leq d(x,y)$. By symmetry we also have the case where $y\in [[\rho,x]]$. Last, when neither $x \in [[\rho,y]]$ nor $y \in [[\rho,x]]$, one just needs to consider $z=x\wedge y$, use the fact that $d(x,y)=d(x,z)+d(y,z)$ and use the previous argument twice.

Let $A$ and $B$ be compact subsets of $\T$ and let $\varepsilon$ such that $A\subset B^{\varepsilon}=\{x\in\T, \exists b\in B, d(x,b)\leq \varepsilon\}$. Let $x\in\pi(A)$ and $a\in A$ such that $x=\pi(a)$ and then let $b\in B$ such that $d(a,b)\leq \varepsilon$. We then have $d(x,\pi(b))\leq\varepsilon$ and thus $\pi(A)\subset \pi(B)^{\varepsilon}$. Reversing the roles of $A$ and $B$ then shows that $d_{\mathcal T,\mathrm H}(\pi(A),\pi(B))\leq d_{\mathcal T,\mathrm H}(A,B)$.

Let $\mu$ and $\nu$ be two probability measures on $\T$ and let $\varepsilon$ such that $d_P(\mu,\nu)\leq\varepsilon$. Let $A$ be a measurable subset of $\T$, we then have $\pi_*\mu(A)=\mu(\pi^{-1}(A))\leq \nu((\pi^{-1}(A))^{\varepsilon})+\varepsilon.$ We also have $(\pi^{-1}(A))^{\varepsilon}\subset\pi^{-1}(A^{\varepsilon})$ and thus $\pi_*\mu(A)\leq\pi_*\nu(A)+\varepsilon$. Reversing the roles of $\mu$ and $\nu$ yields $d_{\mathcal T,\mathrm P}(\pi_*\mu,\pi_*\nu)\leq \varepsilon$.
\end{proof}

Let $Z_{\pi}=\underset{x\in\T}\sup \;d(x,\pi(x))$. This quantity controls all of the difference between $\T$ and $\T'$, even when measured:
 
\begin{lemma}
\label{lem:projcontrol}
We have
	\[Z_{\pi}=\underset{x\in\partial\T'}\sup ht(\T_x),
\]
where $ht(\T_x)=\sup_{y \in \mathcal T_x}d(x,y)$, and
	\[d_{\T,\mathrm{H}}(\T,\T')=Z_{\pi}\]
and, for any measure $\mu$ on $\T$,
	\[d_{\T,\mathrm{P}}(\mu,\pi_*\mu)\leq d_{\T,\mathrm{H}}(\T,\T').
\]
\end{lemma}

\begin{proof} The first point is a direct consequence from the fact that, if $x\in\T'$ then $\pi(x)=x$, while if $x\in\T\setminus\T'$, $x\in\T_{\pi(x)}$. The second point is also a fairly straightforward consequence of the definition of $Z_{\pi}$.
The third point involves simple manipulations of the Prokhorov metric. Let $A$ be a subset of $\T$. Since $A\subset \pi^{-1}(\pi(A))$ and $\pi(A)\subset A^{Z_{\pi}}$, we automatically have $\mu(A)\leq\pi_*\mu(\pi(A))\leq\pi_*\mu(A^{Z_{\pi}})$. On the other hand, we have $\pi^{-1}(A)\subset A^{Z_{\pi}}$, which implies $\pi_*\mu(A)\leq \mu(A^{Z_{\pi}})$.
\end{proof}

\section*{Acknowledgments} We would like thank Nicolas Curien for a stimulating discussion on random trees built recursively, which was the starting point of this work.

\bibliographystyle{siam}
\bibliography{frag}

\end{document}